%% file: main.tex
\documentclass[12pt,a4letter]{amsproc}
\usepackage{mathptmx}
\usepackage{amsthm}
\usepackage{amsmath}
\usepackage{amssymb}
\usepackage{amscd}
\usepackage[latin2]{inputenc}
\usepackage{t1enc}
\usepackage[scr=rsfs]{mathalpha}
\usepackage{indentfirst}
\usepackage{graphicx}
\usepackage{graphics}
\usepackage{pict2e}
\usepackage{epic}
\numberwithin{equation}{section}
\usepackage[margin=1.2in]{geometry}
\usepackage{epstopdf} 
\usepackage{tikz-cd}
\usepackage{mathtools}
\usepackage{cases}
\usepackage[normalem]{ulem}

\usepackage{tikz}
\usepackage{pgfplots}
\pgfplotsset{compat=1.18}
\usetikzlibrary{decorations.pathmorphing}

\makeatletter
\def\l@subsection{\@tocline{2}{0pt}{2.5pc}{5pc}{}}
\def\l@subsubsection{\@tocline{2}{0pt}{2.5pc}{5pc}{}}
\makeatother
\usepackage{biblatex} 
\usepackage{xcolor}
\usepackage{comment}

\usepackage{url}
\usepackage[breaklinks]{hyperref} 

\usepackage{microtype}

\theoremstyle{plain}
\newtheorem{Th}{Theorem}[section]
\newtheorem{Lemma}[Th]{Lemma}
\newtheorem{Cor}[Th]{Corollary}
\newtheorem{Prop}[Th]{Proposition}

 \theoremstyle{definition}
\newtheorem{Def}[Th]{Definition}
\newtheorem{Conj}[Th]{Conjecture}
\newtheorem{Rem}[Th]{Remark}
\newtheorem{?}[Th]{Question}
\newtheorem{Ex}[Th]{Example}

\newcommand{\la}{\langle}
\newcommand{\ra}{\rangle}

\newcommand\norm[1]{\left\lVert#1\right\rVert}

\newcommand\R{\mathbf{R}}

\newcommand{\p}{\partial}
\newcommand{\z}{\overline{z}}

\title[Seiberg-Witten Vortex equations]{On the Existence of Solutions to the Seiberg-Witten Vortex Equations with Exponential Decay on the Plane}

\author{William L. Blair\textsubscript{1}}
\address{Department of Mathematics\\
  The University of Texas at Tyler\\
  Tyler, TX 75799}
\email{wblair@uttyler.edu}

\author{Minh Lam Nguyen\textsubscript{1}}
\address{Department of Mathematics\\
  The University of Texas at Tyler\\
  Tyler, TX 75799}
\email{minhnguyen@uttyler.edu}

\subjclass[2020]{Primary 53Cxx, 57Rxx, 58Jxx, 57Kxx, 30G20, 30Cxx\\1. Department of Mathematics,
  The University of Texas at Tyler,
  Tyler, TX 75799} 

\begin{document}

	\begin{abstract}
		Clifford Taubes showed that the moduli space of the variational equations of the Yang-Mills-Higgs functional on the plane is non-empty, and its elements correspond to "vortices". Inspired by this result, in this paper, we completely characterize exponential decay solutions to the Seiberg-Witten vortex equations, where the equations are derived by applying a Hitchin-type dimensional reduction. We show that the existence of Seiberg-Witten ``vortices'' is not provided by the exponential decay solutions. This is in contrast to the Yang-Mill-Higgs vortex equations.  


	\end{abstract}
\maketitle
\tableofcontents
\begingroup
\let\clearpage\relax
\include{introduction}
\include{preliminaries}

\include{VortexEquations}

\include{Existenceofsolutions}

\include{geometricApplication}
\include{Discussion}
\include{StatementsandDeclarations}
\endgroup

\begin{sloppypar}
\printbibliography
\end{sloppypar}




	
\end{document}

%% file: introduction.tex
\section{Introduction}
\subsection{Main Results}
In this paper, we study some analytic aspects of the solutions of a dimensional reduction to the Seiberg-Witten equations on $\R^2$. The Seiberg-Witten equations were first introduced by Seiberg and Witten in \cite{seiberg1994electric, seiberg1994monopoles}. These equations come from \textit{gauge theory} of Mathematical Physics. Even though they are physically motivated, the precise mathematical interpretation of gauge fields and matter fields as connections and sections on a vector bundle in differential geometry has turned the equations into one of the most instrumental tools in studying problems in low-dimensional topology and symplectic topology (see, e.g., \cite{MR1339810, taubes2000seiberg, kronheimer2007monopoles, MR1367507} for some survey).

The dimensional reduction of the Seiberg-Witten equations we consider here is the one where solutions are assumed to be translational invariant in the last two coordinates of $\R^4$. Effectively, the newly derived equations (which we refer to as the \textit{Seiberg-Witten vortex equations}) consist of objects defined only in two dimensions (cf. \eqref{eq:3.12}, \eqref{eq:maineq}), where they are known to be conformally invariant. Thus, they can also be considered globally on a Riemann surface. Nevertheless, in this paper, we choose not to discuss the global aspect of the equations and focus only on the local analysis. The global aspect of the equations was discussed in Dey's work \cite{MR1952133, MR2698453}, where it is shown that the moduli space of the Seiberg-Witten vortex equations is a symplectic, almost complex manifold that also has a hyperK\"ahler structure.

The derivation of the Seiberg-Witten vortex equations is inspired by the work of Hitchin. In \cite{MR0887284}, Hitchin studied a class of solutions to the (anti)self-dual Yang-Mills equations on $\R^4$ that are translational invariant in the last two coordinates. The resulting equations on $\R^2$ are now called the \textit{Hitchin equations}. The Hitchin moduli space on a Riemann surface has many striking properties, one of which is the correspondence with the moduli space of Higgs bundle \cite{hitchin1987stable}. The exploitation of this correspondence led to many applications in number theory, notably the work of Ng\^o in the Langlands program \cite{ngo2006fibration}. For this reason, it seems interesting and potentially fruitful to apply the same pathway to the Seiberg-Witten equations, as already initiated by Dey's work \cite{MR1952133, MR2698453}.

We refrain from making any more comments about the global geometric aspect of the various gauge theoretic equations mentioned above. From this point on, we only focus on the \textit{local} \color{black} analysis of the Seiberg-Witten vortex equations. 
As a system of PDEs, in plain terms, the Seiberg-Witten vortex equations look for the unknowns $(A_0, A_1,\psi_1,\psi_2)$ that satisfy (cf. \eqref{eq:maineq})
\begin{equation}\label{eq:maineq}
    \begin{cases}
        2\dfrac{\p \psi_2}{\p \overline{z}}+i(A_0+iA_1)\psi_2 = 0,\\
        2\dfrac{\p \psi_1}{\p z}+ i(A_0 - iA_1)\psi_1 = 0,\\
        i\left(\dfrac{\p A_1}{\p x_0} - \dfrac{\p A_0}{\p x_1}\right) = \dfrac{i}{2}(|\psi_1|^2 - |\psi_2|^2),
    \end{cases}
\end{equation}
where $z = x_0 + ix_1$, $x_0$ and $x_1$ are the standard real coordinates, and we identify $\mathbf{C}$ with $\R^2$. Here $A_0, A_1$ are real-valued functions on $\R^2$, and $\psi_1, \psi_2$ are complex-valued functions on $\R^2$. The Seiberg-Witten vortex equations have a symmetry given by the gauge group $\mathcal{G} = Maps(\R^2, U(1))$, i.e., the solutions are invariant under the action of the gauge group $\mathcal{G}$. By moduli space of the equations, we mean the space of solutions quotient out by the $\mathcal{G}$-symmetry. It is not difficult to write down a solution of \eqref{eq:maineq}, albeit "trivial". Here is one: $(\p f/\p x_0, \p f/ \p x_1, 0, 0)$, where $f$ is any smooth function on $\R^2$. In fact, one can be a little bit algebraically creative and realize that

\begin{Th}\label{Th1.1}
    Let $\{z_1,\cdots, z_k\}$ be a finite collection of points on the plane. For any $(c_1,c_2) \in \mathbf{C}^* \times \mathbf{C}^*$, $\{n_\ell\}_{\ell = 1}^k$ a set of positive integers, and $\theta \in \R$, $(A_0, A_1, \psi_1, \psi_2)$ given by
    $$(-2c_2, 0,c_1e^{i\theta}(\overline{z}-\overline{z_1})^{n_1}\cdots (\overline{z}-\overline{z_k})^{n_k}e^{ic_2(z+\overline{z})}, c_1(z-z_1)^{n_1}\cdots(z-z_k)^{n_k} e^{ic_2(z+\overline{z})})$$
    is always a solution of \eqref{eq:maineq}, where $\mathbf{C}^*$ denotes $\mathbf{C}\setminus \{0\}$.
\end{Th}

Denote by $Vor_{\mathfrak{p}}(\mathbf{C})$ the moduli space of solutions to \eqref{eq:maineq} of the type in Theorem \ref{Th1.1}. Consequently, we immediately have the following theorem.

\begin{Th}\label{Thsurjective}
    There is a surjective map $\eta_{\mathfrak{p}}: Vor_{\mathfrak{p}}(\mathbf{C}) \to \bigcup_{n \in \mathbf{N}} Sym^n(\mathbf{C})$. 
\end{Th}

Note that the solutions given in Theorem \ref{Th1.1} and Theorem \ref{Thsurjective} are of polynomial growth, and the "connection" $A= iA_0 dx_0 + iA_1 dx_1$ is always flat. It is natural to ask if there are any other types of solutions to \eqref{eq:maineq} that exhibit different behaviors at infinity and the connection part is not necessarily flat, e.g., exponentially decayed solutions. Aside from the geometric meaning of the equations, the (non)existence of such solutions is an interesting analysis problem in its own right. One of the main results of this paper answers the question of existence of such solutions with non-flat curvature in the negative.

\begin{Def}\label{PropertyE}
     We say a pair of complex-valued functions $(\psi_1, \psi_2)$ defined on $\mathbf{C}$ has property (E) if and only if
    \begin{enumerate}
        \item [(i)] There exists a non-negative function $\lambda: \mathbf{C} \to \mathbf{R}^{\geq0}$ such that $\psi_1 = \lambda \psi_2$.
        \item[(ii)] There exist  \color{black} $M \in (0,\infty)$ \color{black} and $N_1, N_2 \in (0,\infty)$ such that for all $z \in \mathbf{C}$
            \begin{itemize}
                \item [(a)] $0 \leq M - |\psi_1|^2(z) \leq N_1 \exp(-|z|)$
                \item [(b)] $0 \leq M - |\psi_2|^2(z) \leq N_2 \exp(-|z|)$
            \end{itemize}
    \end{enumerate}
\end{Def}

\begin{Th}\label{Th1.2}
   The only smooth solutions that exhibit property (E) of the equation \eqref{eq:maineq} are those with a flat connection component.
\end{Th}

    The exponential decay in Property (E) seems to be out of place. However, it is quite natural to consider. The Seiberg-Witten vortex equations can be thought of as a variant of the vortex equations which were introduced by Ginzburg and Landau to study the theory of superconductivity. From the Mathematics perspective, they are the absolute minimum condition for the Yang-Mills-Higgs functional (cf. Subsection \ref{Sub3.1}). Many authors have studied the existence of solutions related to these types of vortex equations derived from Yang-Mills-Higgs models (see, e.g, \cite{MR0614447, MR0573986, garcia1994direct, bradlow1990vortices, bradlow1991special, MR4516074, MR3435967, MR3513572, ChenChangLiZhang}). We note that in \cite{MR3435967, MR3513572}, the authors' main results are generalizations of Taubes' result \cite{MR0573986,MR0614447} for various gauged-sigma models of the Ginzburg-Landau vortex equations. Using the Yang-Mills-Higgs functional, the solutions of the Ginzburg-Landau vortex equations can be shown to have exponential decay \textit{a priori}. To the best of our knowledge, we do not know whether our equation \eqref{eq:maineq} is also the absolute minimum condition for some Yang-Mills-Higgs-type functional. 
    
    To be clear, we do not claim that all solutions of \eqref{eq:maineq} have Property (E). We show that the only solutions of \eqref{eq:maineq} with Property (E), where $\psi_1$ is nonvanishing, are those with $\psi_1 \equiv \psi_2$, i.e., the connection component is flat. In the case that $\psi_1$ has zeroes, we show there are no solutions of \eqref{eq:maineq} with Property (E). This characterizes all solutions of \eqref{eq:maineq} with Property (E). This is a direct contrast with the Yang-Mills-Higgs vortex equations as considered by Taubes. In that context, solutions must always have exponential decay, the $\mathbf{C}^2$ component has a finite number of zeroes, and the $U(1)$ connection component is not necessarily flat. These zeroes produce ``vortices.'' In our setting, Theorem \ref{Th1.2} shows that no Seiberg-Witten ``vortices'' can be produced from exponential decay solutions. However, ``vortices'' can be produced by solutions with polynomial growth by Theorem \ref{Th1.1}.

    If one restricts to solutions of \eqref{eq:maineq} satisfying Property (E), then the (non)existence problem reduces to (non)existence of a (non)singular $\sinh$-Gordon equation. It is clear that the next two theorems combine to give Theorem \ref{Th1.2}. 
    

\begin{Th}\label{Th1.3}
    Let $M \in (0,\infty)$, and $\{z_1, \cdots, z_n\}$ be any nonempty finite collection of points in the plane. Let $\{\alpha_k\}_{k=1}^n$ be a subset of positive real numbers. The equation $\Delta u =  -2M \sinh(u) +2\pi \sum_{k=1}^{n}\alpha_k\delta(z-z_k)$ has no solution with the condition that $u \to 0$ as $|z| \to \infty$, and $u\geq 0$.
\end{Th}

\begin{Th}\label{introTrivialSoln}
Let $M \in (0,\infty)$. If $u$ is a solution to $-\Delta u = 2M\sinh(u)$ on $\R^2$ where for every $\epsilon >0$, there exists a $C>0$ such that
    \[0\leq u \leq \log(1+C\exp(-\epsilon|z|)) \leq C\exp(-\epsilon|z|),\]
    then $u \equiv 0$.
\end{Th}


Note, Theorems \ref{Th1.3} and \ref{introTrivialSoln} are of independent interest in the field of semilinear equations. The well-known results concerning existence of solutions to semilinear equations on $\mathbf{R}^N$ from the standard references, for example \cite{Lions1, Lions2, Lions3, Atkinson1, Atkinson2, Nirenberg}, do not apply to the equation considered in Theorems \ref{Th1.3} and \ref{introTrivialSoln}. In particular, the theorems found in the above references apply either to dimension $N>2$ or the antiderivative of the semilinear part must have a positive zero. Consequently, the methods from the referenced works also do not apply here. 

To show nonexistence in Theorem \ref{Th1.3}, we prove nonexistence of positive solutions on $\mathbf{R}^2$ with any prescribed, finite, nonempty set of singularities for a semilinear equation where the semilinear part satisfies conditions that, while incompatible with the known results, are satisfied by the $\sinh$ function.

\begin{Th}\label{intro_generalnonexistenceresult}
    Let $f \in C^1(\R)$ such that $f(0) = 0$ and $f'(0) >0$. For any given prescribed non-empty finite set of singularities $\{z_1,z_2,\cdots,z_n\}\subset \R^2$, the equation $-\Delta u = f(u)$ has no positive solution $u$ on $\R^2$ (even in a distributional sense) such that $u \to 0$ as $|z| \to \infty$.
\end{Th}

To show uniqueness in Theorem \ref{introTrivialSoln}, we exploit Pohozaev's identity. Our argument takes advantage of the fact that an antiderivative of the $\sinh$ function is $\cosh - 1$, and $\cosh - 1$ is always nonnegative.

As a final remark, the first two equations of the Seiberg-Witten vortex equations can be thought of as \textit{a system of Vekua-type equations}. The classical Vekua equation is the canonical form of a first-order elliptic equation in the plane and generalizes the classical Cauchy-Riemann equation from complex analysis, which appears naturally in the study of Ginzburg-Landau vortex equations. The connection between Seiberg-Witten vortex equations and Vekua equations is crucial for us to study the zeroes of exponential decay solutions of \eqref{eq:maineq}. We develop the basic theory of solutions to the systems of Vekua-type equations that we require and associate the functions that solve these systems with the solutions of vortex equations. This is the first time that a system of Vekua equations of this kind (to be precise a Vekua equation and the result of applying complex conjugation to both sides of a Vekua equation) has been studied. This connection between the classical complex analysis structure of Vekua equations and the gauge equations of mathematical physics is novel. See \cite{Vek,Bers}, and Section \ref{preliminaries} for background on Vekua equations.

\subsection{Application}
There is a direct application of Theorem \ref{introTrivialSoln} that is of independent geometric interest. The $\sinh$-Gordon equation arises naturally in the study of constant mean curvature (CMC) surfaces. Since the $\sinh$-Gordon equation characterizes solutions of the Seiberg-Witten vortex equations with Property (E), it follows that Property (E)-reducible solutions of the equations, i.e., ones with flat connection component, correspond to copies of the plane in $\mathbf{R}^3$. As a pure geometric statement, we prove that 

\begin{Th}\label{intro_asymptoticallyflatatinfintitymustbeplane}
    Let $\mathbf{x} : \R^2 \to \R^3$ be a complete conformal smooth parametrization of a CMC surface $\Sigma \subset \R^3$ that is also an immersion. Suppose that the Gaussian curvature $K$ is uniformly bounded. Further, suppose for every $\epsilon > 0$, there are $C>0$ such that the conformal factor $E$ satisfies
    \[ 0 \leq E -1 \leq C \exp(-\epsilon|z|).\]
    Then $\Sigma$ must be $\R^2$ up to Euclidean motion. 
\end{Th}

On the surface, the assumption that the conformal factor in Theorem \ref{intro_asymptoticallyflatatinfintitymustbeplane} decays to $1$ exponentially at infinity appears too strong. However, this analytic assumption is crucial. If $\mathbf{x}: \R^2 \to \R^3$ is a smooth parametrization of a surface $\Sigma \subset \R^3$ such that $\mathbf{x}$ is
\begin{enumerate}
    \item complete,
    \item conformal,
    \item an immersion, and 
    \item has constant mean curvature $H$,
\end{enumerate}
then $H = 0$, i.e., $\Sigma$ is a minimal surface in $\R^3$. Indeed, because of choice of orientation, we may assume that $H \geq 0$. If $H >0$, then a theorem of Meeks and Tinaglia shows that $\Sigma$ must be compact \cite{MeeksandTinaglia}. See also Theorem 1.2 in the survey article \cite{meeksandcosurvey}. Since the domain of parametrization $\R^2$ is clearly not compact, it follows that $H > 0$ is impossible. Hence, $H = 0$. Given this, if we additionally assume uniformly bounded Gaussian curvature, it is natural to ask whether there is such a simply connected minimal surface other than the plane. One class of examples that satisfies these conditions is the Enneper surfaces. See example below.

\begin{Ex}[\cite{Nitsche}]\label{Enneper surface}
    An Enneper surface is globally parametrized by
    \begin{equation*}
        \begin{cases}
            x &= \dfrac{1}{3}x_0\left(1-\dfrac{1}{3}x^2_0+x^2_1 \right)\\
            y &= \dfrac{1}{3}x_1\left(1-\dfrac{1}{3}x^2_1+x^2_0 \right)\\
            z &= \dfrac{1}{3}(x^2_0-x^2_1)
        \end{cases}
    \end{equation*}
    An Enneper surface has constant zero mean curvature, and its Gaussian curvature $K$ is given by
    \[K = \dfrac{-4}{9(1+x^2_0+x^2_1)^4}.\]
    It is clear that $|K|$ is uniformly bounded on $\R^2$. See Figure 1.
\end{Ex}


\begin{figure}[htp]
    \includegraphics[width=9cm]{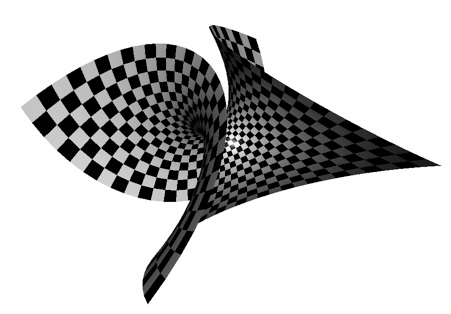}
    \caption{An Enneper Surface}
    \end{figure}

     In light of the above example, Theorem \ref{intro_asymptoticallyflatatinfintitymustbeplane} shows that, for an Enneper surface to ``flatten out'' to become an affine plane, its metric at infinity must be flat. We comment that the exposition above makes use of the theorem of Meeks and Tinaglia \cite{MeeksandTinaglia}. However, the proof of Theorem \ref{intro_asymptoticallyflatatinfintitymustbeplane} will not use their theorem in any way. Instead, we require our uniqueness result Theorem \ref{introTrivialSoln}.  

\subsection{Organization}
The paper is organized as follows. In Section \ref{preliminaries}, we present some material about the classical Vekua equation and prove some new solution representation results about the system of Vekua-type equations that we use later. Section \ref{Sec3} contains background about gauge theory and introduces the terminology of the main results of the paper. In Section \ref{Sec4}, we prove the nonexistence and uniqueness results regarding exponential decay solutions of the Seiberg-Witten vortex equations. In Section \ref{CMC}, we prove the application. In Section \ref{Sec5}, we give some brief comments about the limitations of the technique we used when applied to the situation of the Seiberg-Witten vortex equations \textit{with Higgs fields} (cf. \eqref{eq:3.12}) and discuss extending the techniques to other generalizations of the Seiberg-Witten equations.

\begin{proof}[Acknowledgement]\renewcommand{\qedsymbol}{}
    Parts of this work were carried out during the first-named author's visit to Washington University in St. Louis in the Fall of 2023. The authors would like to thank the university for their generous accommodations. The second-named author is also grateful for many conversations with Aliakbar Daemi regarding various vortex equations on Riemann surfaces. 
\end{proof}

%% file: preliminaries.tex
\section{Preliminaries} \label{preliminaries}

\subsection{Classical Vekua Equation}

In this section, we provide some background about nonhomogeneous Cauchy-Riemann equations that will be used throughout. We work with a bounded simply connected domain $D$ of the complex plane with smooth boundary\color{black}. A nonhomogeneous Cauchy-Riemann equation is any equation of the form 
\begin{equation}\label{eq:nonCR}
    \dfrac{\p w}{\p \z} = f,
\end{equation}
where $f \not\equiv 0$. To study solutions to equations of the form of \eqref{eq:nonCR}, we first recall the classic Cauchy-Pompieu theorem. 

\begin{Th}[$\bar{\p}$-Poincare Lemma \cite{MR0507725}; Cauchy-Pompeiu Theorem, Theorem 20 \cite{BegBook} ]\label{th:cpt}
    Every $w \in C^1(D)\cap C(\overline{D})$ has the representation
    \begin{equation}\label{eq:cpt}
        w(z) = -\dfrac{1}{2\pi} \int_{\p D}\dfrac{w(\zeta)}{\zeta- z}\,d\zeta - \dfrac{1}{\pi} \int_D \dfrac{\dfrac{\p w}{\p\z}(\zeta)}{\zeta- z}\,d\eta\,d\xi,
    \end{equation}
    where $\zeta = \eta + i \xi$.
\end{Th}

Since $w\in C(\p D)$ in the hypothesis of the last theorem, it follows that the contour integral on the right-hand side of \eqref{eq:cpt} is a holomorphic function (see \cite{GK}). By applying $\dfrac{\p}{\p\z}$ to both sides of \eqref{eq:cpt}, we have
\[
    \dfrac{\p w}{\p\z}(z) =  \dfrac{\p}{\p\z}\left(- \dfrac{1}{\pi} \int_D \dfrac{\dfrac{\p w}{\p\z}(\zeta)}{\zeta- z}\,d\eta\,d\xi \right).
\]
Hence, the area integral is a right-inverse to the Cauchy-Riemann operator $\dfrac{\p}{\p\z}$. This behavior persists for rougher classes of functions than those considered in Theorem \ref{th:cpt}, as the next theorem shows. 

\begin{Th}[Theorem 1.16 \cite{Vek}]\label{th:secondkind}
    For $f \in L^1(D)$, every solution of 
    \[
        \dfrac{\p w}{\p\z} = f
    \]
    has the form 
    \[
        w = \varphi + T(f),
    \]
    where 
    \[
        T(f)(z) := - \dfrac{1}{\pi} \int_D \dfrac{f(\zeta)}{\zeta- z}\,d\eta\,d\xi
    \]
    and $\zeta = \eta + i\xi$.
\end{Th}

A well known property of the operator $T(\cdot)$ defined in the last theorem is the following.

\begin{Th}[Theorem 1.19 \cite{Vek}]\label{th:vekH}
    For every $f \in L^q(D)$, $q>2$, $T(f) \in C^{0,\alpha}(\overline{D})$, where $\alpha = \dfrac{q-2}{q}$. 
\end{Th}

A well-studied nonhomogeneous Cauchy-Riemann equation is the Vekua equation
\begin{equation}\label{eq:VekEq}
    \dfrac{\p w}{\p\z} = Aw + B\overline{w},
\end{equation}
where $A, B$ are functions in a Lebesgue space on $D$. Solutions of this equation were classically studied by I. N. Vekua \cite{Vek} (who called them generalized analytic functions) in their study of infinitesimal bendings of surfaces and L. Bers \cite{Bers} (who called them pseudoanalytic functions) in their study of functions that generalize holomorphic functions. The Vekua equation \eqref{eq:VekEq} is an important class of nonhomogeneous Cauchy-Riemann equations because its solutions share many properties of holomorphic functions. This similarity is realized by the following representation formula. 

\begin{Th}[``The Basic Lemma'' \cite{Vek}]\label{Th:simprin}
Every function $w$ that solves
\[
    \dfrac{\p w}{\p\z} = Aw + B\overline{w}
\]
in $D$, where $A,B \in L^q(D)$, $q>2$, has the form
\[
    w = \varphi e^\phi,
\]
where $\varphi$ is holomorphic in $D$ and 
\[
    \phi(z) := \begin{cases}
                    T\left(A+ B\dfrac{\overline{w}}{w}\right)(z), & w(z) \neq 0 \\
                    T\left(A+ B\right)(z), & w(z) = 0
                \end{cases}.
\] 
\end{Th}

This representation is called the ``similarity principle'' or the ``representation of the first kind.''  From this representation, we see that generalized analytic functions inherit their zero set behavior from holomorphic functions, and since $\phi \in C^{0,\alpha}(\overline{D})$ by Theorem \ref{th:vekH}, it follows that $|e^\phi|$ is bounded above and below away from zero, so many other results about holomorphic functions that rely solely on size estimates are recoverable for generalized analytic functions.

Equations of the form 
\[
     \dfrac{\p w}{\p\z} = Aw + B\overline{w} + f,
\]
where $f \not \equiv 0$, are called nonhomogeneous Vekua equations, and by Theorem \ref{th:secondkind}, if $Aw + B\overline{w} + f \in L^1(D)$, then 
\[
    w = \varphi + T(Aw + B\overline{w} + f),
\]
for some holomorphic function $\varphi$. However, we lose the similarity principle representation. In general, there is no reason that $\dfrac{f}{w} \in L^q(D)$, $q>2$, so $T\left(A + B\dfrac{\overline{w}}{w} + \dfrac{f}{w}\right)$ may not converge. In the special case that $\dfrac{f}{w} \in L^q(D)$, $q>2$, and $w(z) \neq 0$, for all $z \in D$, there is the following. 

\begin{Prop}\label{th:nonhomSimPrin}
    For $A, B \in L^q(D)$, $q>2$, any nonvanishing solution $w$ of 
    \[
        \dfrac{\p w}{\p\z} = Aw + B\overline{w} + f,
    \]
    such that $\dfrac{f}{w} \in L^q(D)$, $q>2$, has the representation
    \[
        w = \varphi e^\phi,
    \]
    where $\varphi$ is holomorphic and $\phi$, defined as
    \[
    \phi(z) := T\left(A+ B\dfrac{\overline{w}}{w} + \dfrac{f}{w}\right)(z),    
    \]   
    is in $C^{0,\alpha}(\overline{D})$, $\alpha = \dfrac{q-2}{q}$.
\end{Prop}

\begin{proof}
     Let $w$ be a solution of 
    \[
        \dfrac{\p w}{\p\z} = Aw + B\overline{w} + f,
    \]
    such that $\dfrac{f}{w} \in L^q(D)$, $q>2$. Since $T(\cdot)$ is a right-inverse to $\dfrac{\p}{\p\z}$ by Theorem \ref{th:secondkind}, it follows that 
    \begin{align*}
        \dfrac{\p}{\p\z} \left( \dfrac{w}{e^\phi}\right) &= \dfrac{e^\phi \dfrac{\p w}{\p\z} - w \dfrac{\p}{\p\z}(e^\phi)}{(e^\phi)^2}\\
        &= \dfrac{e^\phi(Aw + B\overline{w} + f) - w e^\phi \left(A + B\dfrac{\overline{w}}{w} + \dfrac{f}{w}\right)}{(e^\phi)^2}\\
        &= 0.
    \end{align*}
    Hence, $\dfrac{w}{e^\phi}$ is holomorphic. Since $A + B\dfrac{\overline{w}}{w} + \dfrac{f}{w} \in L^q(D)$, $q>2$, it follows that $\phi \in C^{0,\alpha}(D)$, by Theorem \ref{th:vekH}
\end{proof}

Note, nonhomogeneous Vekua equations are required to employ our technique to the Seiberg-Witten vortex equations with Higgs field. Since we only have a similarity principle type representation for nonhomogeneous Vekua equations in a very restricted case, we choose not to consider Seiberg-Witten vortex equations with Higgs field in the sequel in favor of focusing on the case of Seiberg-Witten vortex equations without Higgs field.

\subsection{System of Vekua Equations}

Next, we consider systems of the form 
\begin{equation}\label{eq:genSys}
    \begin{cases}
        \dfrac{\p w_1}{\p z} + i A w_1  = 0\\
        &\\
        \dfrac{\p w_2}{\p \z} + i \overline{A} w_2   = 0
    \end{cases},
\end{equation}
where $A$ is a function. 

Note that system \eqref{eq:genSys} is comprised of a Vekua equation and the result of applying complex conjugation to both sides of a Vekua equation. We work to analyze pairs of solutions $(w_1,w_2)$ to systems in the form of \eqref{eq:genSys}.

\begin{Th}\label{Th:systemsimprin}
    Let $A \in L^q(D)$, $q>2$. Every solution pair $(w_1, w_2)$ of the system 
     \[
        \begin{cases}
        \dfrac{\p w_1}{\p z} + i A w_1  = 0\\
        &\\
        \dfrac{\p w_2}{\p \z} + i \overline{A} w_2  = 0
        \end{cases}
    \]
    has the representation 
    \[
        (w_1, w_2) = ( \, \overline{e^{\phi_1}\varphi_1}, e^{\phi_2} \varphi_2) ,
    \]
    where 
    \[
        \phi_1(z) = 
                    T\left(i \overline{A} \right)(z)
    \]
    and 
    \[
        \phi_2(z) = 
                    T\left(-i\overline{A}  \right)(z)
    \]
    are in $C^{0,\alpha}(\overline{D})$ and $\varphi_1, \varphi_2$ are holomorphic.
\end{Th}

\begin{proof}
    The system 
     \[
        \begin{cases}
        \dfrac{\p w_1}{\p z} + i A w_1 = 0\\
        &\\
        \dfrac{\p w_2}{\p \z} + i \overline{A} w_2 +  = 0,
        \end{cases}
    \]
    is equivalent to 
     \[
        \begin{cases}
        \dfrac{\p w_1}{\p z}  = -iAw_1 \\
        &\\
        \dfrac{\p w_2}{\p \z} = -i\overline{A}w_2
        \end{cases}.
    \]
    We now consider each of the equations individually. Observe that if 
\[
    \dfrac{\p w_1}{\p z} = -iAw_1 ,
\]
then 
\[
    \dfrac{\p \overline{w_1}}{\p \z}  =  \overline{-iAw_1 } = i\overline{A}\overline{w}.
\]
Hence, by Proposition \ref{th:nonhomSimPrin}, 
\[
    \overline{w_1} = e^{\phi_1} \varphi_1,
\]
where 
\[
    \phi_1(z) = 
                    T\left(i\overline{A}\right)(z)
\]
and $\varphi_1$ is holomorphic. Thus, 
\[
    w_1 = \overline{e^{\phi_1}\varphi_1}.
\]

Similarly, it follows by Proposition \ref{th:nonhomSimPrin} that 
\[
    w_2 = e^{\phi_2}\varphi_2,
\]
where 
\[
        \phi_2(z) = 
                    T\left(-i\overline{A} \right)(z)
    \]
and $\varphi_2$ is holomorphic. 
\end{proof}

\color{black}

\subsection{Zeroes of Exponential Decay Solutions} \label{Sub2.3}

In this subsection, we show that if solutions to the system of Vekua equations \eqref{eq:genSys} now considered on \textit{the plane} $\mathbf{C}$ have the exponential decay condition, then they must have only a finite number of zeroes which are contained in some closed disk of $\mathbf{C}$.

Since we now consider the system \eqref{eq:genSys} on the plane, we require that the coefficient functions be members of a Lebesgue space specialized for coefficients of Vekua equations on the plane. 

\begin{Def}
For $p\geq 1$ and $D(0,1)$ the complex unit disk centered at the origin, we denote by $L^{p,\nu}(\mathbf{C})$ the set of functions $f: \mathbf{C} \to \mathbf{C}$ such that 
\[
    f \in L^p(D(0,1)) \quad \text{ and } \quad f_\nu(z) := |z|^{-\nu}f(1/z) \in L^p(D(0,1)),
\]
where $\nu$ is a real number. 
\end{Def}

\color{black}

\begin{Lemma}\label{lem: expDecayZeros}
    Let $f$ be a complex-valued function defined on $\mathbf{C}$ such that there exist positive constants $M,N$ and 
    \[
        0 \leq M - |f(z)|^2 \leq N e^{-|z|},
    \]
    for every $z$. There exists a radius $r>0$ so that every zero of the function $f$ is contained in the closed disk of radius $r$.
\end{Lemma}

\begin{proof}
    Suppose for every $R>0$ there exists a $z_R$ such that $|z_R| > R$ and $f(z_R) = 0$. Observe that 
    \[ 
        0 < M \leq N e^{-|z_R|} \leq N e^{-R},
    \]    
    for every $z_R$ and $R$. Since such a $z_R$ exists for every $R>0$, it follows that 
    \[
        0 = \lim_{R\to \infty} N e^{-R} \geq \lim_{R\to \infty} N e^{-|z_R|} \geq M > 0,
    \]
    which is a contradiction.
\end{proof}

\begin{Lemma}\label{lem: simprinzeros}
    Let $\mathfrak{S}\subset\mathbf{C}$ be a bounded, simply connected domain and $A, B \in L^q(\mathfrak{S})$, $q>2$. Every solution $w$ of the Vekua equation
    \[  
        \frac{\p w}{\p\z} = Aw + B\overline{w}
    \]
    has the form $w = h e^\gamma$, where $e^\gamma$ is H\"older continuous and $h$ is a holomorphic function. A complex number $z$ is a zero of $w$ if and only if $z$ is a zero of $h$.
\end{Lemma}

\color{black}

\begin{proof}
    By Theorem \ref{Th:simprin}, every solution of the Vekua equation has the form $w = h e^\gamma$, where $h$ is holomorphic and $e^\gamma$ is H\"older continuous. Since $|e^{\gamma(z)}| > 0$, for all $z$, it follows that if $w(z) = 0$, then $h(z) = 0$.  In the other direction, if $h(z) = 0$, then $w(z) = h(z) e^{\gamma(z)} = 0$. 
\end{proof}

\begin{Lemma}\label{lem: holoFuncZeros}
Every function $f$ that is holomorphic on a closed disk and not identically zero has a finite zero set. 
\end{Lemma}

\begin{proof}
    Suppose that the zero set of $f$ is infinite. Since the closed disk is compact, it follows that there is a subsequence that converges to a point of the closed disk by the Bolzano-Weierstrass Property \cite{Folland}. This implies that $f$ is the identically zero function \cite{GK}, which is a contradiction. 
\end{proof}

We combine the preceding three lemmas to justify the following proposition.

\begin{Prop} \label{Prop:finitezero}
    Let $A, B \in L^{q,2}(\mathbf{C})$, $q>2$. Every function $w$ that solves 
   \[
       \frac{\p w}{\p \z} = Aw + B\overline{w}
   \]
    on $\mathbf{C}$ such that there exist positive constants $M,N$ that satisfy
    \[
        0 \leq M - |w(z)|^2 \leq N e^{-|z|}
    \]
    has a finite zero set. 
\end{Prop}

\begin{proof}
    Since there exist positive constants $M$, $N$ such that 
    \[
        0 \leq M - |w(z)|^2 \leq N e^{-|z|},
    \]
    by Lemma \ref{lem: expDecayZeros}, there exists a closed disk $\overline{D(0,r)}$ of radius $r>0$ such that every zero of $w$ is contained in $\overline{D(0,r)}$. Restricting $w$ to $\overline{D(0,r)}$, then by Lemmas \ref{lem: simprinzeros} and \ref{lem: holoFuncZeros}, $w$ must only have a finite number of zeros.  
\end{proof}

\color{black}

The last proposition will allow us to characterize the zero sets of certain $\psi_1, \psi_2$ that are components of solutions to \eqref{eq:maineq} by associating them with the systems \eqref{eq:genSys}.

\begin{Th}\label{finitezeroessystemvekuaexpdecay}
    Let $A \in L^{q,2}(\mathbf{C})$, $q>2$. Every solution pair $(w_1, w_2)$ of the system 
     \[
        \begin{cases}
        \dfrac{\p w_1}{\p z} + i A w_1  = 0\\
        &\\
        \dfrac{\p w_2}{\p \z} + i \overline{A} w_2  = 0,
        \end{cases}
    \]
    such that there exist positive constants $M_1, N_1, M_2, N_2$ satisfying
    \[
        0 \leq M_1 - |w_1(z)|^2 \leq N_1 e^{-|z|}
    \]
    and 
    \[
        0 \leq M_2 - |w_2(z)|^2 \leq N_2 e^{-|z|},
    \]
    for every $z \in \mathbf{C}$, has a finite zero set. 
\end{Th}

\color{black}

\begin{proof}
Recall from Theorem \ref{Th:systemsimprin} that solutions of this type of system, restricted to a bounded simply connected domain, \color{black} have the form $(w_1, w_2) = (\overline{e^{\gamma_1} }\overline{h_1}, e^{\gamma_2} h_2)$, where the $e^{\gamma_j}$ are H\"older continuous functions and the $h_j$ are holomorphic, $j = 1,2$. The result follows immediately for $w_2$ by Proposition \ref{Prop:finitezero}. The result follows for $w_1$ by applying Proposition \ref{Prop:finitezero} to $\overline{w_1}$ and recognizing that it has the same zero set as $w_1$. 
\end{proof}

\color{black}

%% file: VortexEquations.tex
\section{Vortex Equations}\label{Sec3}
This section presents a brief introduction to mathematical gauge theory and sets up relevant terminologies that go into the statement of the main result of the paper. For more details, we direct the readers to the following references \cite{MR0614447, MR0573986, MR3012377, MR1367507, MR1339810, MR1306021, MR2698453, MR0887284, SACLIOGLU1997675, MR1362874}. The material in this section is well-known to experts. Our intention for providing detailed proofs of some of these included results is for self-containment and accessibility to a larger audience.

\subsection{Yang-Mills-Higgs Gauge Theory}\label{Sub3.1}
From the physical perspective of Yang-Mills-Higgs theory in the $n$-dimensional Euclidean space $\R^{n}$, the variables are given by
\begin{itemize}
    \item A \textit{gauge potential} $A = \sum_{j=0}^{n-1} A_j(x)dx_j$. In differential geometry, gauge potentials are to be understood as connections of certain principal $G$-bundle $P$ over $\R^n$, where $G$ is a Lie group. Thus, $A_j$ are functions defined on $\R^n$ that take values in the Lie algebra $\mathfrak{g}$ of $G$.
    \item A \textit{matter field} $\phi = \phi(x)$. Once again, from the mathematical perspective, $\phi$ should be thought of as a section of an associated vector bundle $E = P \times_{G} V \to \R^n$ that is defined as long as one has a representation $\rho: G \to GL(V)$, where $V$ is some finite dimensional vector space.  
\end{itemize}
Here, $x$ denotes a point in $\R^n$, written in terms of the standard coordinate system as $x = (x_0, \cdots, x_{n-1})$. Let $\{e_0, \ldots, e_{n-1}\}$ be the standard orthonormal basis of $\R^n$. $G$ is called the \textit{gauge group}, i.e, it is the group of transformations of the \textit{internal symmetry space} $E$ that $\phi$ takes value in. The interaction between a gauge potential and matter field is via the notion of taking a covariant derivative. In particular, 
$$\nabla_A \phi = \displaystyle \sum_{j=0}^{n-1} (d\phi(e_j) + \rho(A_j)(\phi)) dx_j.$$
One should think of $\nabla_A \phi$ as an $E$-valued $1$-form defined on $\R^n$. A connection $A$ also determines for us the notion of curvature $F(A)$, which locally can be written as
$$F(A) = dA + A \wedge A = \displaystyle \sum_{j,k} \dfrac{1}{2}\left(\dfrac{\partial A_k}{\partial x_j} - \dfrac{\partial A_j}{\partial x_k} + [A_j, A_k]\right) dx_j \wedge dx_k.$$

In what follows, we consider the case where $n =2$, the gauge group $G = U(1)$, $P$ is taken to be the trivial principal bundle $P = \R^2 \times U(1)$, and $\rho: U(1) \to GL(\mathbf{C})$ to be the standard representation of $U(1)$ on $\mathbf{C}$ given by the (complex) scalar multiplication. As a result, the associated vector bundle $E$ simplifies to be the trivial complex line bundle $\mathbb{L}$. Since we are working on Euclidean space $\R^2$ where there is no interesting topology and the Lie algebra of $U(1)$ is simply $i\R$, various actors defined above can be simplified as follows. The matter field $\phi$ is now simply a $\mathbf{C}$-valued function on $\R^2$. The gauge potential $A$ is a purely imaginary-valued $1$-form on $\R^2$, written as $A = iA_0 dx_0 + iA_1 dx_1$. Hence, the curvature $F(A)$ would just be the curl of $A$ given by
$$F(A) = i \left( \dfrac{\partial A_1}{\partial x_0} - \dfrac{\partial A_0}{\partial x_1} \right) dx_0 \wedge dx_1.$$
Whereas, $\nabla_A \phi$ is now a complex valued $1$-form on $\R^2$.

Denote by $\mathcal{C} = i\Omega^{1}(\R^2) \times C^{\infty}(\R^2, \mathbf{C})$ the configuration space. Let $dVol$ be the standard Lebesgue measure on $\R^2$. The \textit{Euclidean Yang-Mills-Higgs action functional} on $\mathcal{C}$ is given by
\begin{equation}\label{eq:YMH}
    \mathscr{Y}(A, \phi) = \dfrac{1}{2}\int_{\R^2} \left(\dfrac{1}{4}|F(A)|^2 + |\nabla_A \phi|^2 + (|\phi|^2-1)^2\right)\, dVol.
\end{equation}
The various point-wise norms that appear in the integrand above deserve some justification. Firstly, since $\phi$ is complex-valued, $|\phi|^2 = \phi \overline{\phi}$. Next, since $\nabla_A \phi = \sum_{j=0}^{1} (d\phi(e_j) + i A_j \phi)dx_j$ is a $\mathbf{C}$-valued $1$-form on $\R^2$, we simply define $|\nabla_A \phi|^2$ to be the sum of the squares of the norm of the $\mathbf{C}$-component of the form. Similarly,
$$|F(A)|^2 =  \left|\dfrac {\partial A_1}{\partial x_0} - \dfrac{\partial A_0}{\partial x_1} \right|^2.$$

For convenience, from now on, we shall write $\partial_j := \partial / \partial x_j$.

\begin{Lemma}[Bogomolny]\label{completion of squares}
    Let $\phi$ be written as $\phi_0 + i \phi_1$ when we view it as a complex-valued function defined on $\R^2$ or $\begin{pmatrix} \phi_0 & \phi_1\end{pmatrix}^T$ when viewed as a map from $\R^2 \to \R^2$, where $\phi_j$ are $\R$-valued functions. We can re-write the Yang-Mills-Higgs action functional \eqref{eq:YMH} as follows
    \begin{align*}
        \mathscr{Y}(A,\phi) = \dfrac{1}{2}\int_{\R^2} & [(\partial_0\phi_0 - A_0\phi_1)+(\partial_1\phi_1 + A_1 \phi_0)]^2 + \\
        & + [(\partial_0\phi_1+A_0\phi_0) - (\partial_1\phi_0 - A_1 \phi_1)]^2 + \\
        & + \dfrac{1}{2}\int_{\R^2} ((\partial_0 A_1 - \partial_1 A_0)/2 +|\phi|^2-1)^2 + \dfrac{1}{2}\int_{\R^2} F(A) + \\
        & + \int_{\R^2} d(\phi_1(d\phi_0 - iA\phi_1) - \phi_0(d\phi_1 + iA\phi_0)).
    \end{align*}
\end{Lemma}

\begin{proof}
     When we view $\phi = \begin{pmatrix} \phi_0 & \phi_1 \end{pmatrix}^T : \R^2 \to \R^2$, its total derivative in matrix form is written as 
     $$d\phi = \begin{pmatrix} \partial_0 \phi_0 & \partial_1 \phi_0 \\ \partial_0 \phi_1 & \partial_1 \phi_1\end{pmatrix}.$$
     As a result, $d\phi(e_0) = \partial_0 \phi_0 + i \partial_0 \phi_1$ and $d\phi(e_1) = \partial_1 \phi_0 + i \partial_1 \phi_1$. Thus, from the definition of the Yang-Mills-Higgs functional, we can re-write
     \begin{align*}
     \mathscr{Y}(A,\phi) = \dfrac{1}{2} \int_{\R^2} &\dfrac{1}{4}  (\partial_0 A_1 - \partial_1 A_0)^2 + |(\partial_0\phi_0 - A_0 \phi_1) + i (\partial_0 \phi_1 + A_0 \phi_0)|^2 +\\
     &+ |(\partial_1\phi_0 - A_1 \phi_1) + i (\partial_1\phi_1 + A_1 \phi_0)|^2 + (|\phi|^2-1)^2.
     \end{align*}
     By completion of squares, we can re-arrange the above as
     \begin{align*}
         \mathscr{Y}(A,\phi) = &\dfrac{1}{2} \int_{\R^2}  [(\partial_0\phi_0 - A_0\phi_1)+(\partial_1\phi_1 + A_1 \phi_0)]^2 + \\
        & + \dfrac{1}{2}\int_{\R^2} [(\partial_0\phi_1+A_0\phi_0) - (\partial_1\phi_0 - A_1 \phi_1)]^2 + \\
        & + \dfrac{1}{2}\int_{\R^2} ((\partial_0 A_1 - \partial_1 A_0)/2 +|\phi|^2-1)^2 + \dfrac{1}{2}\int_{\R^2} F(A) + \\
        & + \dfrac{1}{2}\int_{\R^2} -(\partial_0 A_1 - \partial_1 A_0) |\phi|^2 + \\
        &+ \dfrac{1}{2}\int_{\R^2}- 2(\partial_0\phi_0 - A_0\phi_1)(\partial_1 \phi_1 + A_1 \phi_0) + \\
        &+ \dfrac{1}{2}\int_{\R^2} 2 (\partial_0 \phi_1 + A_0 \phi_0)(\partial_1\phi_0 - A_1\phi_1).
     \end{align*}
     We simplify the integrand of the last three terms of the equation above as follows
     \begin{align*}
        & -(\partial_0 A_1 - \partial_1 A_0) |\phi|^2 - 2(\partial_0\phi_0 - A_0\phi_1)(\partial_1 \phi_1 + A_1 \phi_0) + 2 (\partial_0 \phi_1 + A_0 \phi_0)(\partial_1\phi_0 - A_1\phi_1)\\
        & = \underbrace{(-\partial_0 A_1 \cdot \phi^2_0 - 2 A_1 \phi_1 \cdot \partial_0 \phi_0)}_{\text{$-\partial_0(A_1 \phi^2_0)$}} + \underbrace{(-\partial_0 A_1 \cdot \phi^2_1 - 2A_1 \phi_1 \cdot \partial_0 \phi_1)}_{\text{$-\partial_0(A_1\phi^2_1)$}} + \\
        & + \underbrace{(\partial_1 A_0 \cdot  \phi^2_0 + 2A_0 \phi_0 \cdot \partial_0 \phi_1)}_{\text{$\partial_1(A_0\phi^2_0)$}} + \underbrace{(\partial_1 A_0 \cdot \phi^2_1 + 2A_0 \phi_1 \cdot \partial_1 \phi_1)}_{\text{$\partial_1(A_0\phi^2_1)$}} + \\
        &+\underbrace{2(-\partial_0\phi_0 \cdot \partial_1 \phi_1 + \partial_0\phi_1 \cdot \partial_1\phi_0)}_{\text{$-2\det (d\phi)$}}.
     \end{align*}
     Note that
     \begin{align*}
         d(\phi_0 d\phi_1 - \phi_1 d\phi_0) &= 2d\phi_0 \wedge d\phi_1 = 2\det(d\phi) dx_0 \wedge dx_1,\\
         d(iA|\phi|^2) = - d(A_1|\phi|^2dx_1 + A_0|\phi|^2 dx_0)& = ( -\partial_0(A_1 |\phi|^2) + \partial_1( A_0 |\phi|^2))dx_0 \wedge dx_1.
     \end{align*}
     Combine all of the above and we obtain the new formula for $\mathscr{Y}(A,\phi)$ as claimed.
\end{proof}

\begin{Prop}\label{lowerboundforYMH}
    Let $(A, \phi)$ such that $A$ is a continuous connection and $\phi \in C^1$. Suppose we have
    $$\displaystyle \lim_{R\to \infty} \sup_{|x|=R} |1-|\phi|^2| = 0, \quad |x|^{1+\delta}|\nabla_A \phi| \leq c_0,$$
    for some $\delta, c_0 > 0$. Then, $\mathscr{Y}(A, \phi) \geq \dfrac{1}{2}\displaystyle \int_{\R^2} F(A)$. Equality happens if and only if $(A,\phi)$ satisfies
    \begin{equation}\label{eq:Taubesvortex}
        \begin{cases}
            \dfrac{\partial \phi}{\partial \overline{z}} - \dfrac{i}{2}(A_0+iA_1)\phi = 0,\\
            \dfrac{i}{2}F(A) = (1-|\phi|^2)dz\wedge d \overline{z},
        \end{cases}
    \end{equation}
    where $A = iA_0 dx_0 + i A_1 dx_1$.
\end{Prop}

\begin{proof}
    Note that $\phi_1(d\phi_0 - iA \phi_1) - \phi_0 (d\phi_1 + i A \phi_0)$ can be re-written as
    $$\left \la \begin{pmatrix} d & -iA \\ iA & d \end{pmatrix}\, \begin{pmatrix} \phi_0 \\ \phi_1 \end{pmatrix}, \begin{pmatrix} \phi_1 \\ -\phi_0 \end{pmatrix} \right \ra = \Re \la \nabla_A \phi, -i\phi \ra,$$
    where $\langle \cdot, \cdot \rangle$ is the standard Euclidean dot product.
    As a result, by integration by parts and the Cauchy-Schwarz inequality, we have
    \begin{align*}
    & \displaystyle \int_{\R^2} d(\phi_1(d\phi_0 - iA\phi_1) - \phi_0(d\phi_1 + iA\phi_0)) \\
    &= \lim_{R \to \infty} \int_{|x|\leq R} d(\phi_1(d\phi_0 - iA\phi_1) - \phi_0(d\phi_1 + iA\phi_0))\\
    &= \lim_{R \to \infty} \int_{|x|=R} \Re\la\nabla_A \phi, -i\phi\ra\\
    &\leq \lim_{R \to \infty} \int_{|x|=R} |\nabla_A \phi| \cdot |\phi|.
    \end{align*}
    Let $c_R$ denote the supremum of $|1-|\phi|^2|$ on $|x|=R$. Then $|\phi| \leq \sqrt{1 + c_R}$ for any $|x|=R$. By the hypothesis of the proposition, the last integral on the right-hand side of the above inequality can be estimated further by
    \begin{align*}
        \lim_{R \to \infty} \int_{|x|=R} |\nabla_A \phi| \cdot |\phi| \leq \lim_{R\to\infty} \int_{|x|=R} \dfrac{c_0}{R^{1+\delta}}\cdot \sqrt{1+c_R} = \lim_{R\to \infty} \dfrac{2\pi c_0}{R^\delta}\cdot \sqrt{1+c_R} = 0.
     \end{align*}
    Therefore, the integral of $d(\phi_1(d\phi_0 - iA\phi_1) - \phi_0(d\phi_1 + iA\phi_0))$ over $\R^2$ is equal to zero. Hence, by Lemma \ref{lowerboundforYMH}, we immediately get the estimate of $\mathscr{Y}$ as claimed. The statement of equality can be checked directly via calculations (also, see \cite{MR0614447}, Ch.3).
\end{proof}

One can say more.

\begin{Th}[Proposition 3.5, Theorem 1.1 \cite{MR0614447}]\label{Th:Taubesexistence}
    Suppose $(A,\phi)$ satisfies the condition of Proposition \ref{lowerboundforYMH} and $F(A) \in L^1$. Then
    \begin{enumerate}
        \item $\displaystyle \dfrac{1}{2\pi}\int_{\R^2} F(A) = N$, where $N$ is an integer.
        \item If $N \geq 0$, given a set $\{z_j\}_{j=1,\cdots N}$ in $\mathbf{C}$, there is a finite action solution $(A,\phi)$ (i.e., $\mathscr{Y}(A,\phi) < \infty$) of \eqref{eq:Taubesvortex} such that
        \begin{enumerate}
            \item $(A,\phi)$ is globally smooth.
            \item The zeroes of $\phi$ are $\{z_j\}$. And as $z \to z_j$, we have $\phi(z,\overline{z}) \sim c_j(z-z_j)^{n_j}$, where $c_j \neq 0$ and $n_j$ is the multiplicity of $z_j$.
        \end{enumerate}
    \end{enumerate}
\end{Th}

\begin{Rem}
    \eqref{eq:Taubesvortex} is called a \textit{vortex equation}. Solutions of the vortex equation \eqref{eq:Taubesvortex} are also solutions of the Euler-Lagrange equation of $\mathscr{Y}$. Theorem 1.2 in \cite{MR0614447} shows that the only finite action critical point of $\mathscr{Y}$ is a solution of \eqref{eq:Taubesvortex} in the form in Theorem \ref{Th:Taubesexistence}. Note the first equation in \eqref{eq:Taubesvortex} is a Vekua equation (cf. \eqref{eq:VekEq}).
\end{Rem}

\begin{Rem}
    Both Lemma \ref{completion of squares} and Proposition \ref{lowerboundforYMH} are stated in \cite{MR0614447}, but a proof of them was not given in detail. We present it here for the sake of self-containment.
\end{Rem}

There is a symmetry of $\mathcal{C}$ that makes $\mathscr{Y}$ invariant. The symmetry is given by the gauge group $\mathcal{G} = Maps(\R^2, U(1))$, where the action $\mathcal{G} \curvearrowright \mathcal{C}$ is given by
$$(\sigma, (A,\phi)) \mapsto (A + \sigma d\sigma^{-1}, \sigma \cdot \phi).$$
By direct calculations, it is not difficult to see that $\mathscr{Y}(\sigma \cdot (A, \phi)) = \mathscr{Y}(A, \phi)$. Thus, $\mathscr{Y}$ descends to a function (also denoted by the same name when the context is clear) $\mathscr{Y} : \mathcal{C}/\mathcal{G} \to \R$. As a result, solutions of the vortex equation \eqref{eq:Taubesvortex} are also $\mathcal{G}$-invariant.

\subsection{Seiberg-Witten Gauge Theory}
There is another variant of the vortex equation \eqref{eq:Taubesvortex} that is derived from a slightly different perspective. For that, we make a detour to dimension four and briefly discuss a gauge theoretic equation called the \textit{Seiberg-Witten equations}. The Seiberg-Witten equations can be defined on any $4$-manifold. However, following the theme of the previous subsection, we mainly focus on its formulation in the Euclidean space $\R^4$. 

Consider the standard flat metric on $\R^4$, $\{x_0,\cdots,x_3\}$ are the coordinates, $\{e_0,\cdots,e_3\}$ are the standard orthonormal basis of its tangent bundle $T\R^4 = \R^4 \times \R^4$, $\{dx_0,\cdots,dx_3\}$ are the dual bases for $T^*\R^4$. Fix the constant $spin^c$ structure $\rho : \mathbf{H} = \R^4 \to \mathbf{C}^{4\times 4}$ (see \cite{MR3012377, MR1367507}) defined by 

$$\rho(v) = \begin{pmatrix} 0 & v \\ -\overline{v}^t & 0 \end{pmatrix}, \, \, \, \, v = \begin{pmatrix} a+bi & c+di\\ -c+di & a-bi \end{pmatrix}.$$
So we identify $e_0 = Id, e_1 = I, e_2 = J,$ and $e_3 = K$ with
$$I = \begin{pmatrix} i & 0 \\ 0 & -i\end{pmatrix},\, \, \, J = \begin{pmatrix} 0 & 1\\-1 & 0\end{pmatrix},\,\,\, K = \begin{pmatrix} 0 & i \\ i & 0\end{pmatrix}.$$
Let $S^+$ denote the above $spin^c$ structure and $L_{\rho} = \R^4 \times \mathbf{C}$ be the associated line bundle. Consider the $spin^c$ connection $\nabla = \nabla_A$ given by
$$\nabla_j \, \psi = \dfrac{\partial\psi}{\partial x_j} + iA_j\,\psi,\,\,\,j=0,\cdots,3,$$
where $A_j : \R^4 \to \R$ and $\psi : \R^4 \to \mathbf{C}^2$. The associated connection on $L_{\rho}$ is given by $A = iA_0 \,dx_0 + iA_1\, dx_1 + iA_2\, dx_2 + iA_3\,dx_3$. Note that $\psi: \R^4 \to \mathbf{C}^2$ is called a \textit{spinor} and $\nabla_A$ is called a \textit{spinor connection}. We denote $\mathcal{C}(S^+) = i\Omega^1(\R^4) \times C^{\infty}(\R^4, \mathbf{C}^2)$ by the configuration of the $spin^c$ structure $S^+$.

\begin{Def}\label{EuclideanDiracop}
    Given a spinor connection $A$ on $\R^4$. The \textit{Dirac operator} $D^+_A$ defined on $C^\infty(\R^4, \mathbf{C}^2)$ is an elliptic first order operator given by
    $$D^+_A \psi = -\nabla_0 \psi + I \nabla_1 \psi + J \nabla_2 \psi + K \nabla_3 \psi.$$
\end{Def}

In a general Euclidean space $\R^n$ with the standard flat metric, there is the Hodge $\star$-operator that takes a $p$-form to a $(n-p)$-form. It is defined as follows. Let $\omega$ be a $p$-form on $\R^n$ written in Einstein summation notation as
$$\omega = \omega_{j_1 \cdots j_p} dx_{j_1}\wedge \cdots \wedge dx_{j_p}.$$
Then
$$\star \omega = \dfrac{1}{p!} \epsilon^{k_1\cdots k_p j_1 \cdots j_{n-p}}\omega_{k_1\cdots k_p} dx_{j_1} \wedge \cdots \wedge dx_{j_{n-p}},$$
where $\epsilon^{k_1\cdots k_p}$ is the totally anti-symmetric tensor and $\epsilon^{1\cdots n} = 1$. In dimension four, $\star$-operator turns a $2$-form to another $2$-form. Since $\star^2 = 1$, its eigenvalues are $\pm 1$. We say that a two form $\omega$ is \textit{(anti) self-dual} if and only if $\star \omega = \pm \omega$. Any two form $\omega$ can be written as a sum of a self-dual form and an anti-self-dual form, 
$$\omega =  \underbrace{\omega^+}_{\text{$\dfrac{1}{2}(\omega + \star \omega)$}} + \underbrace{\omega^-}_{\text{$\dfrac{1}{2}(\omega - \star \omega)$}}.$$
Another special feature of dimension four is that $\rho$ is also an isometry between the space of purely imaginary self-dual $2$-forms on $\R^4$ and the space of all self-adjoint traceless endomorphisms of $\mathbf{C}^2$. In particular, the self-dual part $F^+(A)$ of the curvature of $A$ can be written as
\begin{equation*}
F^+(A) = i(F_{01}+F_{23})I + i(F_{02}+F_{31})J + i(F_{03}+F_{12})K,\, \, \text{ where } F_{jk} = \dfrac{\partial A_k}{\partial x_j} - \dfrac{\partial A_j}{\partial x_k}.
\end{equation*}

There is another way to obtain a self-adjoint traceless endomorphism of $\mathbf{C}^2$ from a spinor $\psi$. We define $\mu(\psi)$ as a linear combination of $I, J, K$ in the following way
$$\mu(\psi) = \dfrac{1}{2} (\psi^* I \psi) I + \dfrac{1}{2} (\psi^* J \psi) J + \dfrac{1}{2} (\psi^* K \psi) K.$$
Here $\psi^*$ denotes the conjugate transpose of $\psi$ when we view it as a column vector in $\mathbf{C}^2$. Having set these up, we are ready to write down the Seiberg-Witten equations in $\R^4$.

\begin{Def}\label{SWeqn}
    The Seiberg-Witten equations on $\R^4$ is a system of non-linear elliptic PDEs that look for the unknown $(A,\psi) \in \mathcal{C}(S^+)$ satisfying
    \begin{equation}\label{eq:SW}
        \begin{cases}
            D^+_A \psi = 0\\
            i(F_{01}+F_{23}) = \dfrac{1}{2}\psi^* I \psi\\
            i(F_{02}+F_{31}) = \dfrac{1}{2}\psi^* J \psi\\
            i(F_{03}+F_{12}) = \dfrac{1}{2}\psi^* K \psi.
        \end{cases}
    \end{equation}
\end{Def}

\begin{Rem}\label{explicitcalculation}
    If we write $\psi = \begin{pmatrix} \psi_1 & \psi_2 \end{pmatrix}^T$, where $\psi_j: \R^4 \to \mathbf{C}$, then note that $\psi^* I \psi = i(|\psi_1|^2-|\psi_2|^2)$, $\psi^* J \psi = 2i \Im (\overline{\psi_1}\psi_2)$, and $\psi^* K \psi = 2 i\Re (\overline{\psi_1}\psi_2)$. These expressions are all homogeneous polynomials in $\psi_1, \psi_2$ variables of degree $2$.
\end{Rem}

\subsection{Hitchin dimensional reduction of the Seiberg-Witten equations}
The following dimensional reduction of \eqref{eq:SW} from $\R^4$ to $\R^2$ is in the spirit of Hitchin's work on self-dual Yang-Mills equations on Riemann surfaces \cite{MR0887284}. It has been done also in Dey's thesis \cite{MR2698453, MR1952133}. We will assume that $(A,\psi) \in \mathcal{C}(S^+)$ is invariant in the $x_2, x_3$-coordinate. Let's first take a look at the curvature equations in \eqref{eq:SW}. In this setup, we have
\begin{equation}\label{eq:curvatureeqn}
\begin{cases}
iF_{01} = \dfrac{i}{2}(|\psi_1|^2 - |\psi_2|^2)\\ i\left(\dfrac{\partial A_2}{\partial x_0} - \dfrac{\partial A_3}{\partial x_1}\right) = i\Im(\overline{\psi_1}\psi_2) \\ i\left(\dfrac{\partial A_3}{\partial x_0} + \dfrac{\partial A_2}{\partial x_1}\right) = i\Re(\overline{\psi_1}\psi_2).
\end{cases}
\end{equation}

If we denote $\phi_0 = iA_2$ and $\phi_1 = iA_3$, then we can rewrite \eqref{eq:curvatureeqn} as
\begin{equation}\label{eq:rewritecurvatureeqn}
\begin{cases}
iF_{01} = \dfrac{i}{2}(|\psi_1|^2 - |\psi_2|^2)\\ \dfrac{\partial \phi_0}{\partial x_0} - \dfrac{\partial \phi_1}{\partial x_1} = i\Im(\overline{\psi_1}\psi_2) \\ \dfrac{\partial \phi_0}{\partial x_1} + \dfrac{\partial \phi_1}{\partial x_0} = i\Re(\overline{\psi_1}\psi_2).
\end{cases}
\end{equation}

Introduce the complex coordinate $z = x_0 + x_1 \, i$ and note that $\partial/\partial z = (\partial/\partial x_0 - i \partial/\partial x_1)/2$ and $\partial/\partial \overline{z} = (\partial /\partial x_0 + i \partial/\partial x_1)/2$, combine the last two equations of \eqref{eq:rewritecurvatureeqn} to obtain
\begin{equation}\label{eq:3.6}
\dfrac{\partial\phi}{\partial \overline{z}} = -\dfrac{1}{2}\psi_1\overline{\psi_2},\, \, \text{where } \phi = \phi_0 + i \phi_1.
\end{equation}

If we let $\Phi = \phi dz - \overline {\phi} d\overline{z}$, then we can re-write \eqref{eq:curvatureeqn} in forms as
\begin{equation}\label{eq:3.7}
\begin{cases}
2\overline{\partial}\,\Phi = -\dfrac{1}{2}\psi_1\overline{\psi_2} \, dz \wedge d\overline{z}\\
F(A) = \dfrac{i}{2}(|\psi_1|^2 - |\psi_2|^2) \, dz\wedge d\overline{z}
\end{cases}
\end{equation}

Moving on to the Dirac equation of \eqref{eq:SW}, in our setup, we view $\psi = \begin{pmatrix} \psi_1 & \psi_2\end{pmatrix}^T$, where $\psi_1, \psi_2 : \mathbf{C} \to \mathbf{C}$. Then $D^+_A \psi = 0$ can be re-written as follows
\begin{align}
\begin{pmatrix}\dfrac{\partial}{\partial x_0} \psi_1 + iA_0\psi_1\\ \dfrac{\partial}{\partial x_0} \psi_2 + iA_0 \psi_2\end{pmatrix}-\begin{pmatrix}i\dfrac{\partial}{\partial x_1}\psi_1 - A_1 \psi_1 \\ -i\dfrac{\partial}{\partial x_1}\psi_2 + A_1\psi_2\end{pmatrix}-\begin{pmatrix}iA_2\psi_2\\-iA_2\psi_1\end{pmatrix} - \begin{pmatrix}-A_3\psi_2\\-A_3\psi_1\end{pmatrix} &= 0 \nonumber \\
 \begin{pmatrix} \dfrac{\partial}{\partial x_0} \psi_1 + iA_0 \psi_1 - i\dfrac{\partial}{\partial x_1} \psi_1 + A_1 \psi_1 - iA_2\psi_2 + A_3 \psi_2 \\ \dfrac{\partial}{\partial x_0} \psi_2 + iA_0\psi_2 + i \dfrac{\partial}{\partial x_1} \psi_2 - A_1 \psi_2 + iA_2 \psi_1 + A_3 \psi_1\end{pmatrix} &= 0  \nonumber \\
 \begin{pmatrix} 2\dfrac{\partial}{\partial z} + i (A_0 - iA_1) & -\phi \\ -\overline{\phi} & 2 \dfrac{\partial}{\partial \overline{z}} + i(A_0 + i A_1) \end{pmatrix}\cdot \begin{pmatrix} \psi_1 \\ \psi_2 \end{pmatrix} &= 0 \label{eq:3.8}
\end{align}

If we let $A^{1,0} = i(A_0 - iA_1)/2 \, dz$ and $A^{0,1} = i(A_0 + iA_1)\,d\overline{z}$ so that $A = iA_0 dx_0 + iA_1 dx_1 = A^{1,0} + A^{0,1}$, then we can re-write \eqref{eq:3.8} one more time as:
\begin{equation}\label{eq:3.9}
\begin{pmatrix} \dfrac{1}{2} \Phi^{0,1} & -(\overline{\partial} + A^{0,1}) \\ \partial + A^{1,0} & -\dfrac{1}{2}\Phi^{1,0}\end{pmatrix}\cdot \begin{pmatrix} \psi_1 \\ \psi_2 \end{pmatrix} = 0.
\end{equation}
\eqref{eq:3.7} and \eqref{eq:3.9} together gives us the dimensional reduction of the Seiberg-Witten equations \eqref{eq:SW} over $\R^2$
\begin{equation}\label{eq:3.10}
\begin{cases}
2\overline{\partial}\,\Phi = -\dfrac{1}{2}\psi_1\overline{\psi_2} \, dz \wedge d\overline{z}\\
F(A) = \dfrac{i}{2}(|\psi_1|^2-|\psi_2|^2) \, dz\wedge d\overline{z}\\
\begin{pmatrix} \dfrac{1}{2} \Phi^{0,1} & -(\overline{\partial} + A^{0,1}) \\ \partial + A^{1,0} & -\dfrac{1}{2}\Phi^{1,0}\end{pmatrix}\cdot \begin{pmatrix} \psi_1 \\ \psi_2 \end{pmatrix} = 0. 
\end{cases}
\end{equation}

Solutions of \eqref{eq:3.10} are $(A, \psi_1,\psi_2, \Phi)$, where $A\in i\Omega^1(\R^2)$, $\Phi \in \Omega^{1,1}(\R^2)$, and $\psi_1, \psi_2 : \mathbf{C} \to \mathbf{C}$. One should think of $A$ as some associated connection to the trivial complex line bundle over $\R^2$, and $\Phi$ is a \textit{Higgs field}. Without the Higgs field, there is another variant of the vortex equations given by
\begin{equation}\label{eq:3.11}
\begin{cases}
F(A) = \dfrac{i}{2}(|\psi_1|^2 - |\psi_2|^2) \, dz\wedge d\overline{z}\\
\begin{pmatrix} 0 & -(\overline{\partial} + A^{0,1}) \\ \partial + A^{1,0} & 0\end{pmatrix}\cdot \begin{pmatrix} \psi_1 \\ \psi_2 \end{pmatrix} = 0. 
\end{cases}
\end{equation}

The "non-forms" version of \eqref{eq:3.10} will be read as
\begin{equation}
\begin{cases}\label{eq:3.12}
\dfrac{\partial \phi}{\partial\overline{z}} = -\dfrac{1}{2} \psi_1 \overline{\psi_2},\\
i\left(\dfrac{\partial A_1}{\partial x_0} - \dfrac{\partial A_0}{\partial x_1}\right) = \dfrac{i}{2}(|\psi_1|^2 - |\psi_2|^2)\\
2\dfrac{\partial \psi_2}{\partial \overline{z}}+i(A_0 + iA_1)\psi_2 - \overline{\phi}\psi_1 = 0\\
2\dfrac{\p \psi_1}{\p z}+ i(A_0 - iA_1)\psi_1 - \phi\psi_2 = 0
\end{cases} 
\end{equation}

The gauge group action on a configuration is defined similarly as at the end of Subsection \ref{Sub3.1}. The space of solutions of these equations quotient out by $\mathcal{G}$ are called the \textit{moduli spaces of the Seiberg-Witten vortex equations}. 

\begin{Prop}[cf. Theorem \ref{Th1.1}]\label{Prop3.9}
    There are non-trivial solutions to \eqref{eq:3.11}. Explicitly, for any $(c_1,c_2) \in \mathbf{C}^*\times \mathbf{C}^*$, $(A_0, A_1, \psi_1, \psi_2) = (-2c_2, 0, \pm c_1 e^{ic_2(z+\overline{z})}, c_1 e^{ic_2(z+\overline{z})})$ is a solution of \eqref{eq:3.11}.    
\end{Prop}

\begin{proof}
    The proof is just a direct calculation.
\end{proof}

Note that the solutions above have no zeroes in $\mathbf{C}$. Suppose now we would like to use Proposition \ref{Prop3.9} as a building block for solutions with prescribed zeroes. The process would start as follows. If we wish to use $\psi_2$ as in Proposition \ref{Prop3.9} and obtain another solution, by the Vekua representation result \ref{Th:simprin}, we have to rescale $\psi_2$ by a holomorphic function $h_2$. Keeping $A_0, A_1$ the same as in the above proposition, then $(A_0, A_1, \psi_2) = (-2c_2, 0, c_1 h_2 e^{ic_2(z+\overline{z})})$ is a solution of $2\p_{\overline{z}}\psi_2 + i(A_0 + iA_1) \psi_2 = 0$ if and only if $\p_{\overline{z}}h_2 = 0$. Similarly, we also have to rescale $\psi_1$ by an anti-holomorphic $h_1$. For $(A_0, A_1, \psi_1) = (-2c_2, 0, c_1  h_1 e^{ic_2(z+\overline{z})})$ to be a solution of $2\partial_{z}\psi_1 + i(A_0 - iA_1)\psi_1 = 0$, we need $\p_z (h_1) = 0$. Now, for $(A_0, A_1, \psi_1, \psi_2) = (-2c_2, 0 , c_1 h_1 e^{ic_2(z+\overline{z})}, c_1 h_2 e^{ic_2(z+\overline{z})})$ to be a solution of the curvature equation, we need $| h_1| = |h_2|$. To sum up,  we have 

\begin{Cor}\label{Cor3.10}
    $(A_0, A_1, \psi_1, \psi_2) = (-2c_2, 0 , c_1 h_1 e^{ic_2(z+\overline{z})}, c_1 h_2 e^{ic_2(z+\overline{z})})$ is another solution of \eqref{eq:3.11} if and only if
\begin{equation}\label{eq:3.13}
    \begin{cases}
        \p_{\overline{z}}h_2 = 0,\\
        \p_{z}(h_1) = 0,\\
        |h_1| = |h_2|.
    \end{cases}
\end{equation}
\end{Cor}

Since $\overline{h_1}, h_2$ are holomorphic functions that share the same modulus, there is a constant $\lambda \in S^1$ such that $\overline{h_1} = \lambda h_2$. As a result, this gives us a recipe to yield other solutions. Let $\{z_1,\cdots,z_k\}$ be some points in $\mathbf{C}$. Let $\theta$ be any real number and consider
$$h_2(z) = (z-z_1)^{n_1}\cdots(z-z_k)^{n_k}, \quad h_1(z) = e^{i\theta}(\overline{z}-\overline{z_1})^{n_1}\cdots (\overline{z}-\overline{z_k})^{n_k}.$$
Then $(A_0, A_1, \psi_1, \psi_2)$ given by
\begin{align}\label{eq:3.14}
(-2c_2, 0,c_1e^{i\theta}(\overline{z}-\overline{z_1})^{n_1}\cdots (\overline{z}-\overline{z_k})^{n_k}e^{ic_2(z+\overline{z})}, c_1(z-z_1)^{n_1}\cdots(z-z_k)^{n_k} e^{ic_2(z+\overline{z})}).
\end{align}
is always a solution of \eqref{eq:3.11}. As a final remark, such a solution always has a flat connection by construction. Later, we will show the existence of a different kind of solution to \eqref{eq:3.11} (where the connection is not necessarily flat) with a finite set of zeroes via a monotone method\color{black}. Denote by $Vor_{\mathfrak{p}}$ the moduli space of solutions to \eqref{eq:3.11} of type \eqref{eq:3.14}. We summarize the discussion above in the form of the following theorem.

\begin{Th}[cf. Theorem \ref{Thsurjective}]\label{Th3.11}
    There is a surjective map $\eta_{\mathfrak{p}}: Vor_{\mathfrak{p}}(\mathbf{C}) \to \bigcup_{n \in \mathbf{N}} Sym^n(\mathbf{C})$. 
\end{Th}

%% file: Existenceofsolutions.tex
\section{Existence Of Solutions With Exponential Decay}\label{Sec4}

\subsection{A $\sinh$-Gordon Equation}\label{Gordansinh}
We recall the setup for one of our main results on the existence of exponential decay solutions to \eqref{eq:maineq}. 

\begin{Def}[cf. Definition \ref{PropertyE}]
    We say a pair of complex-valued functions $(\psi_1, \psi_2)$ defined on $\mathbf{C}$ has property (E) if and only if
    \begin{enumerate}
        \item [(i)] There exists a non-negative function $\lambda: \mathbf{C} \to \mathbf{R}^{\geq0}$ such that $\psi_1 = \lambda \psi_2$.
        \item[(ii)] There exist  \color{black} $M \in (0,\infty)$ \color{black} and $N_1, N_2 \in (0,\infty)$ such that for all $z \in \mathbf{C}$
            \begin{itemize}
                \item [(a)] $0 \leq M - |\psi_1|^2(z) \leq N_1 \exp(-|z|)$
                \item [(b)] $0 \leq M - |\psi_2|^2(z) \leq N_2 \exp(-|z|)$
            \end{itemize}
    \end{enumerate}
\end{Def}

The main theorem of this section that we prove is the following.

\begin{Th}[cf. Theorem \ref{Th1.2}]\label{property(E)existence}
    The only smooth solutions that exhibit property (E) of the equation \eqref{eq:maineq} are those with a flat connection component. 
\end{Th}

\color{black}
    Notice we are not claiming all solutions with a flat connection component of \eqref{eq:maineq} have property (E). In the literature, solutions of \eqref{eq:maineq} with flat connection component are called \textit{reducible}. The purpose of this theorem is to demonstrate that the collection of all solutions of \eqref{eq:maineq} with property (E) is a subset of the collection of all reducible solutions. 
    
    All solutions with property (E) can be placed in one of two classes: either $\psi_1$ never vanishes (which implies that $\lambda$ and $\psi_2$ also never vanish) or $\psi_1$ vanishes somewhere on the plane. In order to prove Theorem \ref{property(E)existence}, 
    \begin{itemize}
        \item Firstly, we show that there cannot be any solution with property (E) where $\psi_1$ has zeroes on the plane.
        \item Secondly, we show that if $(A_0, A_1,\psi_1,\psi_2)$ is a solution with property (E) to \eqref{eq:maineq} and $\psi_1$ never vanishes, then $A= iA_0 dx_0 + iA_1 dx_1$ is flat.
    \end{itemize}
   With that in mind, \color{black} let $(A_0, A_1, \psi_1, \psi_2)$ be a solution of \eqref{eq:maineq} such that $(\psi_1, \psi_2)$ has property (E). For convenience, we restate what \eqref{eq:maineq} is
\[
    \begin{cases}
        2\dfrac{\p \psi_2}{\p \overline{z}}+i(A_0+iA_1)\psi_2 = 0,\\
        2\dfrac{\p \psi_1}{\p z}+ i(A_0 - iA_1)\psi_1 = 0,\\
        i\left(\dfrac{\p A_1}{\p x_0} - \dfrac{\p A_0}{\p x_1}\right) = \dfrac{i}{2}(|\psi_1|^2 - |\psi_2|^2).
    \end{cases}
\]
Since $(\psi_1,\psi_2)$ satisfies a system of Vekua equations and has property (E), by Theorem \ref{finitezeroessystemvekuaexpdecay}, we know that 
$\psi_1, \psi_2$ have a finite number of zeroes on the plane. From the relation $\psi_1 = \lambda \psi_2$, it must be the case that the set $S$ containing the zeroes of $\psi_1$ also contains the zeroes of $\psi_2$ and $\lambda$. Thus, the number of zeroes of $\lambda$ is finite. 

\color{black}
We let $\alpha = \dfrac{1}{2}(A_0-iA_1)$. Then by direct calculations, we can rewrite \eqref{eq:maineq} as
\begin{equation}\label{eq:rewrite1.1}
    \begin{cases}
        \p_{\overline{z}}\psi_2 + i\overline{\alpha} \psi_2 = 0,\\
        \p_{z}\psi_1+i\alpha\psi_1 = 0,\\
        2(\p_{z}\overline{\alpha}-\p_{\overline{z}}\alpha) = \dfrac{i}{2}(|\psi_1|^2-|\psi_2|^2).
    \end{cases}
\end{equation}
In the first equation of \eqref{eq:rewrite1.1}, we solve for $\alpha$ to obtain
\begin{equation}\label{eq:eqforalphapsi2}
    i\overline{\alpha} = -\p_{\overline{z}} \log \psi_2, \quad \psi_2 \neq 0.
\end{equation}
By substituting the expression of $\alpha$ into the second equation of \eqref{eq:rewrite1.1} and conjugate, we have
\begin{equation}\label{eq:holomorphicforlogpsi1psi2}
    \p_{\overline{z}} \log(\overline{\psi_1}\psi_2) = 0, \quad \psi_1, \psi_2 \neq 0.
\end{equation}

We would like to comment that \eqref{eq:holomorphicforlogpsi1psi2} can only make sense at the points in $\mathbf{C}$ where $\psi_1$ does not vanish. Now, since $(\psi_1, \psi_2)$ has property (E), we can rewrite \eqref{eq:holomorphicforlogpsi1psi2} further as following
\[\p_{\overline{z}} \log(\lambda |\psi_2|^2) = 0.\]
By the open mapping theorem, we have $\log (\lambda |\psi_2|^2) = C$, for some constant $C$. In fact, we can specify what $C$ is. Note that
\[|\psi_2|^2 = \dfrac{\exp(C)}{\lambda}.\]
So, by the exponential decay condition of $\psi_1, \psi_2$, we must have
\[M =  \exp(C), \quad \text{ which implies } \quad C = \log M.\]
Given this, as we substitute the expression for $\alpha$ in \eqref{eq:eqforalphapsi2} into the remaining equation of the system \eqref{eq:rewrite1.1}, we obtain
\begin{align}
\Delta \log |\psi_2|^2 &= |\psi_1|^2 - |\psi_2|^2 \nonumber \\
\Delta \log\left(\dfrac{\exp(C)}{\lambda}\right)  &= |\psi_2|^2(|\lambda|^2-1) \nonumber \\
-\Delta \log \lambda  &= \dfrac{1}{\lambda}\cdot \exp(C) (|\lambda|^2-1) \nonumber \\
 \lambda\Delta \log \lambda &= - M \cdot \lambda^2+ M \nonumber
\end{align}
For convenience, we set $u = \log \lambda$. The above equation now reads as
\begin{equation}\label{eq:nonsingulargordonsinh}
    \Delta u = -M (e^u - e^{-u}) = -2M \sinh(u)
\end{equation}
Note that \eqref{eq:nonsingulargordonsinh} only makes sense on $\mathbf{C}\setminus S$. Let $\delta$ be the Dirac distribution. To account for the singularities, we consider the equation in the distributional sense. Without loss of generality, we assume that the set of all singularities of $u$ is exactly $S$. Recall that we set $u = \log \lambda$ as above. This means that the set of zeroes of $\lambda$ is exactly $S$.

\begin{Prop}\label{distributionalgordonsinh}
    Let $S = \{z_1,\cdots, z_n\} \subset \mathbf{R}^2$. Suppose the vanishing order of $\lambda$ at each $z_k$ is $\alpha_k$, where $\alpha_k>0$. If $u =\log \lambda$ satisfies \eqref{eq:nonsingulargordonsinh} away from $S$, then on the entire plane $\mathbf{R}^2$, in the distributional sense, $u$ satisfies
        \begin{equation}\label{eq:singulargordonsinh}
    \Delta u =  -2M \sinh(u) + 2\pi \sum_{k=1}^{n}\alpha_k\delta(z-z_k).
\end{equation}
The equation \eqref{eq:singulargordonsinh} is called the singular $\sinh$-Gordon equation
\end{Prop}

\begin{proof}
    Let $\phi \in C^{\infty}_c(\mathbf{R}^2)$ be a smooth compactly supported function on the plane. For $u$ satisfying \eqref{eq:nonsingulargordonsinh} away from $S$, we have
    \[\la \Delta u, \phi\ra = \la u, \Delta \phi\ra = \int_{\mathbf{R}^2} u \Delta \phi = \lim_{\epsilon \to 0}\int_{\R^2\setminus \cup_{k=1}^{n}D(z_k,\epsilon)} u\Delta \phi.\]
    Note that by the Green's identity, we have
    \[\int_{\R^2\setminus \cup_{k=1}^{n}D(z_k,\epsilon)} u\Delta \phi = \int_{\R^2\setminus \cup_{k=1}^{n}D(z_k,\epsilon)} \phi \Delta u + \sum_{k=1}^{n}\int_{\partial D(z_k,\epsilon)}\left(u \dfrac{\partial \phi}{\partial \nu} - \phi \dfrac{\partial u}{\partial \nu} \right)\;ds,\]
    where $\nu$ is the outward normal direction vector to the boundary of  $\R^2\setminus \cup_{k=1}^{n}D(z_k,\epsilon)$. In this case, in polar coordinates $(r,\theta)$, $\partial/\partial \nu = -\partial/\partial r$.

    For convenience, we set
    \[I_0(\epsilon) := \int_{\R^2\setminus \cup_{k=1}^{n}D(z_k,\epsilon)}\phi \Delta u,\]
    
    \[I^{\phi}_k(\epsilon) := \int_{\partial D(z_k,\epsilon)}u \dfrac{\partial \phi}{\partial \nu}\;ds,\] 
    
    \[I^{u}_k(\epsilon) := \int_{\partial D(z_k, \epsilon)} - \phi \dfrac{\partial u}{\partial \nu}\;ds.\]
    We analyze each of these terms in the following.

    For $I_0(\epsilon)$, note that since $u$ satisfies \eqref{eq:nonsingulargordonsinh} on $\R^2\setminus \cup_{k=1}^{n} D(z_k, \epsilon)$,
    \[\lim_{\epsilon \to 0} I_0(\epsilon) = \lim_{\epsilon \to 0}\int_{\R^2 \setminus\cup_{k=1}^{n} D(z_k, \epsilon)} \phi\cdot (-2M)\sinh(u) = \la -2M\sinh (u), \phi\ra. \]
    
    Now, near each $z_k$, we have $\lambda\sim e^{w_k(z)}|z-z_k|^{\alpha_k} $, where $w_k(z)$ is a smooth function. Thus, near $z_k$, we have
    \begin{equation}\label{eq:approxunearzk}
        u \sim \alpha_k \log|z-z_k| + w_k(z).
    \end{equation}
    So, using the local characterization of $u$ near $z_k$ in \eqref{eq:approxunearzk}, for $I^\phi_k(\epsilon)$, in polar coordinates, we have
    \begin{align*}
        I^{\phi}_k(\epsilon) & = -\int_0^{2\pi}u(\epsilon, \theta)\cdot \dfrac{\partial \phi}{\partial r}|_{r=\epsilon}\cdot \epsilon d\theta \\
        & = -\int_0^{2\pi} \left(\alpha_k \log \epsilon + w_k(z_k) + O(\epsilon)\right)\cdot \dfrac{\partial \phi}{\partial r}|_{r=\epsilon}\cdot \epsilon d\theta \\
        & = -\int_0^{2\pi}\alpha_k\log \epsilon \cdot \dfrac{\partial \phi}{\partial r}|_{r=\epsilon}\cdot \epsilon d\theta - \int_{0}^{2\pi}w_k(z_k)\dfrac{\partial \phi}{\partial r}|_{r=\epsilon}\cdot \epsilon d\theta-\int_0^{2\pi} O(\epsilon)\dfrac{\partial \phi}{\partial r}|_{r=\epsilon}\cdot \epsilon d\theta
    \end{align*}
    Since $\phi$ is smooth with compact support, $\partial \phi/\partial r$ is bounded; the second and third terms of the above expression tend to zero as $\epsilon \to 0$, respectively. The first term can be bounded by
    \[\left | -\int_0^{2\pi}\alpha_k\log \epsilon \cdot \dfrac{\partial \phi}{\partial r}|_{r=\epsilon}\cdot \epsilon d\theta \right| \lesssim 2\pi \alpha_k \cdot \epsilon \cdot |\log \epsilon| \to 0, \quad \text{as } \epsilon \to 0.\]
    Hence, we must have $\lim_{\epsilon \to 0} I^{\phi}_k(\epsilon) = 0$.
    
    Lastly, for $I^u_{k}(\epsilon)$, we have
    \begin{align*}
        I^{u}_{k}(\epsilon) &= \int_{0}^{2\pi} \phi(\epsilon, \theta)\cdot \dfrac{\partial u}{\partial r}|_{r=\epsilon}\cdot \epsilon d\theta \\
        & = \int_{0}^{2\pi}\phi(\epsilon,\theta) \cdot \left(\dfrac{\alpha_k}{\epsilon} + O(1)\right)\cdot \epsilon d\theta \\
        & = \int_0^{2\pi}\alpha_k \phi(\epsilon,\theta)\;d\theta + \int_0^{2\pi}\phi(\epsilon, \theta) O(1)\epsilon\;d\theta.
    \end{align*}
    The second term of the above expression goes to zero as $\epsilon \to 0$. The first term approaches $2\pi \alpha_k \phi(z_k)$. Thus, $\lim_{\epsilon\to 0} I^u_k(\epsilon) = \la 2\pi\alpha_k \delta(\bullet-z_k), \phi\ra$. Now, gather everything and we have, for every $\phi \in C^{\infty}_c(\R^2)$,
    \[\la \Delta u, \phi\ra = \la -2M\sinh(u),\phi\ra + \sum_{k=1}^{n} \la 2\pi\alpha_k \delta(\bullet-z_k), \phi\ra.\]
    Equivalently, this means that \eqref{eq:singulargordonsinh} holds in the distributional sense.
\end{proof}

Note that $u = \log \lambda \to \log\sqrt{M/M}=0$ as $|z| \to \infty$. Furthermore, we can figure out the appropriate rate of decay of $u$ so that a solution of \eqref{eq:singulargordonsinh} produces a solution $(A_0,A_1,\psi_1\psi_2)$ of \eqref{eq:maineq} with property (E), i.e., $|\psi_1|^2, |\psi_2|^2$ exponentially converge to $M$. Indeed, if $e^u=\lambda \geq 1$ and for every $\epsilon >0$,  there is a $C>0$ such that
\[0\leq\lambda - 1\leq C\exp(-\epsilon|z|),\]
then from $\lambda |\psi_2|^2 = M$ we have
\[0\leq M - |\psi_2|^2 = (\lambda - 1)|\psi_2|^2 \leq C|\psi_2|^2\exp(-\epsilon|z|)\leq CM\exp(-\epsilon|z|).\]
An analogous estimate for $\psi_1$ can be similarly derived. In terms of $u$, this means that we need to look for solutions of \eqref{eq:singulargordonsinh} where
\begin{equation}\label{eq:asymptoticconditionforu}
    0\leq u \leq \log(1+C\exp(-\epsilon|z|)) \leq C\exp(-\epsilon|z|)
\end{equation}
The discussion above shows that the existence of solutions $(A_0, A_1, \psi_1, \psi_2)$ with property (E) (cf. Theorem \ref{property(E)existence}) is equivalent to the existence of solutions to the $\sinh$-Gordon equation of type \eqref{eq:singulargordonsinh} satisfying \eqref{eq:asymptoticconditionforu}. Below, we show that: 
\begin{itemize}
    \item When $n > 0$, there is no solution to \eqref{eq:singulargordonsinh} satisfying \eqref{eq:asymptoticconditionforu} (see Theorem \ref{existenceofsolutionsinhgordon} for an even stronger statement).
    \item When $n=0$ (i.e., $u=\log \lambda$ has no singularity), then the only smooth solution to \eqref{eq:nonsingulargordonsinh} satisfying \eqref{eq:asymptoticconditionforu} is $u\equiv 0$ (see Theorem \ref{trivialsolution}).
\end{itemize}
To say it differently, if one restricts to solutions of \eqref{eq:maineq} satisfying the additional ansatz encoded in property (E) where the vanishing order of all zeroes is greater than or equal to zero, then the connection component must be flat, and there is no other possibility. \color{black} With this in mind, for the rest of this section, we prove the following statements.

\begin{Th}[cf. Theorem \ref{Th1.3}]\label{existenceofsolutionsinhgordon}
Let \color{black} $M \in (0,\infty)$, and $\{z_1, \cdots, z_n\}$ be any nonempty finite collection of points in the plane. Let $\{\alpha_k\}_{k=1}^n$ be a subset of positive real numbers. The equation $\Delta u =  -2M \sinh(u) +2\pi \sum_{k=1}^{n}\alpha_k\delta(z-z_k)$ has no solution with the condition that $u \to 0$ as $|z| \to \infty$, and $u\geq 0$.
\end{Th}

Theorem \ref{existenceofsolutionsinhgordon} implies that solutions of \eqref{eq:maineq} with property $(E)$ can only come from solutions of \eqref{eq:nonsingulargordonsinh} satisfying \eqref{eq:asymptoticconditionforu}. We show that the only such solution on the plane is the trivial solution.

\begin{Th}[cf. Theorem \ref{introTrivialSoln}]\label{trivialsolution}
    Let $M \in (0,\infty)$. If $u$ is a solution to $-\Delta u = 2M\sinh(u)$ on $\R^2$ where for every $\epsilon >0$, there exists a $C>0$ such that
    \[0\leq u \leq \log(1+C\exp(-\epsilon|z|)) \leq C\exp(-\epsilon|z|),\]
    then $u \equiv 0$.
\end{Th}

Theorem \ref{existenceofsolutionsinhgordon} and Theorem \ref{trivialsolution} combine to yield Theorem \ref{property(E)existence}. The remaining subsections are devoted to the proof of the above theorems.

\subsection{Proof of Theorem \ref{existenceofsolutionsinhgordon}} We prove a more general statement than our situation. The problem eventually reduces to an ordinary differential inequality.

\begin{Th}[cf. Theorem \ref{intro_generalnonexistenceresult}]\label{generalnonexistenceresult}
    Let $f \in C^1(\R)$ such that $f(0) = 0$ and $f'(0) >0$. For any given prescribed non-empty finite set of singularities $\{z_1,z_2,\cdots,z_n\}\subset \R^2$, the equation $-\Delta u = f(u)$ has no positive solution $u$ on $\R^2$ (even in a distributional sense) such that $u \to 0$ as $|z| \to \infty$.
\end{Th}

\begin{proof}
    We prove by contradiction. Suppose otherwise that there is such a $u$ solving the equation $-\Delta u = f(u)$. Then $u$ is a $C^2$ function away from some finite number of singularities $S$. Let $R_S >0$ be large enough such that $\overline{D}(0,R_S) \supset S$.  

    Since $f'(0) > 0$, there is a $\delta >0$ such that for all $0<t<\delta$, 
    $$f(t) > \dfrac{f'(0)}{2}t := ct.$$
    On the other hand, because $u$ decays to zero at infinity, there is a $R_\delta >0$ large enough such that for all $|z|> R_{\delta}$, $0<u<\delta$. Thus, for all $|z| > R_{\delta}$, we have
    \[f(u) > c\cdot u.\]
    Hence, for all $z \in \R^2\setminus \overline{D}(0,R)$, where $R = \max \{R_S, R_\delta\}$,
    \begin{equation}\label{eq:pdinequality}
        -\Delta u > c\cdot u, \quad \quad u>0, \quad \quad \lim_{|z|\to \infty}u(z) = 0.
    \end{equation}

    In polar coordinates, for $r > R$, we define
    \[\overline{u}(r) := \dfrac{1}{2\pi}\int_0^{2\pi} u(re^{i\theta})\;d\theta.\]
    Note that from \eqref{eq:pdinequality}, for all $r>R$, we must have   \begin{equation}\label{eq:differentialordinaryinequalityinu}
        \overline{u}''(r) + \dfrac{1}{r}\overline{u}'(r) + c\overline{u}(r) = \Delta \overline{u}+c\overline{u} < 0, \quad \overline{u} > 0.
    \end{equation}
    We set $w(r)r^{-1/2} = \overline{u}(r)$. A brief calculation shows that \eqref{eq:differentialordinaryinequalityinu} is equivalent to the following
    \begin{equation}\label{eq:differentialinequalityinw}
        r^{-1/2}\left(w'' + \left(c+\dfrac{1}{4r^2} \right)w \right)< 0.
    \end{equation}
    Since $w >0$, \eqref{eq:differentialinequalityinw} implies that
    \begin{equation}\label{eq:lastdiffineq}
    w'' + cw < 0, \quad \quad \text{for all } r>R.
    \end{equation}

    Consider the following boundary value problem,
    \[\beta'' + c\beta = 0, \quad \beta(a)=0, \quad \beta(b) = 0, \]
    here $R< a < b$ and $b-a = \pi/\sqrt{c}$. There is a solution to the problem given by $\beta(r) = \sin(\sqrt{c}(r-a))$. Integration by parts gives us
    \begin{align*}
        \int_a^b \beta w''&=(\beta w')|_{a}^b-\int_a^b \beta' w' = (\beta w')|_a^b - (\beta' w)|_a^b + \int_a^b \beta''w\\
        & = -\beta'(b)w(b) + \beta'(a)w(a) + \int_a^b -c\beta\cdot w \\
        & = \sqrt{c}(w(b) + w(a)) + \int_a^b -c\beta \cdot w.
    \end{align*}
    As a result,
    \[\int_a^b \beta (w''+cw) = \sqrt{c}(w(b)+w(a)).\]
    Note that $\beta(r) = \sin(\sqrt{c}(r-a)) > 0$ for all $r \in (a,b)$. By \eqref{eq:lastdiffineq}, $w''+cw < 0$ for $R<a<r<b$. Hence, the left-hand side of the above integral result is negative. On the other hand, the right-hand side of the identity is positive. This leads to a contradiction, and the proof is complete. 
\end{proof}

\begin{proof}[Proof of Theorem \ref{existenceofsolutionsinhgordon}]
    The proof of this theorem is now straightforward. We apply Theorem \ref{generalnonexistenceresult} to $f(t) = 2M\sinh(t)$ and obtain the desired statement.
\end{proof}

\subsection{Proof of Theorem \ref{trivialsolution}}
The proof of Theorem \ref{trivialsolution} essentially relies on Pohozaev's identity, which originally appeared in \cite{pohozaev} (see also \cite{pllionshberestycki}). The version of the identity we recall below follows from \cite{pllionshberestycki}.

\begin{Prop}[Proposition 1 in Section 2 of \cite{pllionshberestycki}]\label{poho}
    Suppose $f$ is a continuous function from $\R \to \R$ such that $f(0) = 0$. Let $F(t) = \displaystyle \int_0^t f(s)\;ds$. Let $u$ be a function on $\R^N$, where $N \geq 2$, such that
    \[ u \in L^\infty_{loc}(\R^N), \quad \nabla u\in L^2(\R^N), \quad F(u) \in L^1(\R^N),\]
    and $u$ satisfies
    \[-\Delta u = f(u) \text{ on } \R^N.\]
    Then $u$ satisfies
    \[\dfrac{N-2}{2N}\int_{\R^N} |\nabla u|^2 = \int_{\R^N} F(u).\]
\end{Prop}

An immediate consequence of the above identity is that when $N = 2$, we must have
\[\int_{\R^2} F(u) = 0.\]
We exploit this consequence in our situation. To set things up, we establish that if $u$ is a solution to the $sinh$-Gordon equation on the plane and $u$ satisfies \eqref{eq:asymptoticconditionforu}, then $u$ satisfies the hypothesis of Proposition \ref{poho}.

\begin{Lemma}\label{exponentialdecayimpliesLp}
   Let $u$ be a non-negative function on $\R^2$. Suppose that for every $\epsilon > 0$, there exists a $C>0$ such that 
   $$0\leq u \leq \log(1+C\exp(-\epsilon|z|)).$$
   Then $u \in L^\infty_{loc}(\R^2)$. 
\end{Lemma}

\begin{proof}
    This is clear because $\log(1+C\exp(-\epsilon|z|)) \in L^\infty_{loc}(\R^2)$.
\end{proof}

Next we show that if a solution $u$ decays to zero at an exponential rate, then $\nabla u \in L^2$.

\begin{Lemma}\label{exponentialdecayimpliessobolev}
    If $u$ is a solution of $-\Delta u = 2M\sinh(u)$ on $\R^2$ and $u$ satisfies the hypothesis of Lemma \ref{exponentialdecayimpliesLp}, then $\nabla u \in L^2(\R^2)$. 
\end{Lemma}

\begin{proof}
    If $u$ satisfies the equation on $\R^2$, then $u$ also satisfies the equation on $D(z,1)$, where $z$ is any point on the plane. By the interior gradient estimate (see Section 3.4 in \cite{GilbargTrudinger}) applied to $D(z,1)$, we have
    \[|\nabla u(z)| \lesssim \norm{u}_{L^\infty(D(z,1))} + \norm{\sinh(u)}_{L^\infty(D(z,1))}.\]
    Since $u$ approaches zero at infinity and $|\sinh t| = \sinh |t|$, for a fixed $\mathfrak{c} >0$, when $|z|$ is large enough in terms of $\mathfrak{c}$, we have $|\sinh(u)| \leq \mathfrak{c}|u|$. This implies that for $|z|$ large enough, we have
    \[\norm{\sinh(u)}_{L^\infty(D(z,1))} \leq \mathfrak{c}\norm{u}_{L^\infty(D(z,1))}.\]
    As a result, for large $z$, we must have
    \[|\nabla u(z)| \lesssim \norm{u}_{L^\infty(D(z,1))}.\]
    Since $u$ has an exponential decay rate, from the above we obtain that $|\nabla u|$ also has an exponential decay rate. As a result, $\nabla u \in L^2(\R^2)$. 
\end{proof}

Lastly, we prove that a certain non-negative function in terms of $u$ is integrable on $\R^2$. 

\begin{Lemma}\label{L1ofFu}
    Let $u$ satisfy the hypothesis of Lemma \ref{exponentialdecayimpliesLp}. Then $\cosh(u) - 1 \in L^1(\R^2)$. 
\end{Lemma}

\begin{proof}
    It suffices to show that for large $R>0$, we must have
    \[\int_{\R^2\setminus \overline{D}(0,R)} \cosh(u) - 1 < \infty.\]
    Since $|u|$ approaches zero at infinity, for a fixed $\mathfrak{c}>0$, there exists an $R_\mathfrak{c} >0$ such that for all $|z| > R_{\mathfrak{c}}$,
    \[ \cosh (u) - 1 \leq \mathfrak{c} u.\]
    As a result,
    \[\int_{\R^2\setminus \overline{D}(0,R_\mathfrak{c})} \cosh(u) - 1 \leq \mathfrak{c}\int_{\R^2\setminus \overline{D}(0,R_\mathfrak{c})} u \leq \mathfrak{c}\cdot C\int_{{\R^2\setminus \overline{D}(0,R_\mathfrak{c})}} \exp(-\epsilon|z|).\]
    The integral on the right-hand side of the above estimate is finite. We obtain the result as claimed. 
\end{proof}

\begin{proof}[Proof of Theorem \ref{trivialsolution}]
    Let $f(t) = 2M \sinh(t)$. Note that $F(t) = 2M(\cosh t - 1)$. By Lemma \ref{exponentialdecayimpliesLp}, Lemma \ref{exponentialdecayimpliessobolev}, and Lemma \ref{L1ofFu}, we apply Proposition \ref{poho} to our setting, if such a solution to the $\sinh$-Gordon equation on the plane exists, then
    \[\int_{\R^2} \cosh u - 1 = 0.\]
    But $\cosh t \geq 1$ for any $t\in \R$. The equality happens only when $t = 0$. Combined with the result of the integral above, we must have $u \equiv 0$. 
\end{proof}

\color{black}

%% file: geometricApplication.tex
\section{A Geometric Application}\label{CMC}

In this section, we consider a direct consequence of Theorem \ref{trivialsolution} for surfaces in $\R^3$ with constant mean curvature. Firstly, we briefly review some basic notions of differential geometry for surfaces. Most of the preliminary is standard and can be found in greater detail in \cite{eisenhart, hopf, docarmosurface, wente, spruck}. Below, we mostly follow \cite{wente}.  

\subsection{Fundamental Equations on Surfaces}
Let $\Omega \subset \R^2$ be a domain. We say that a smooth  parametrization $\mathbf{x} : \Omega \to \R^3$ of a surface $\Sigma \subset \R^3$ is \textit{conformal} if and only if its first fundamental form satisfies
\[ I  = d\mathbf{x} \cdot d\mathbf{x} = E(dx_0^2+dx_1^2),\]
where recall the convention we use in the paper for the coordinates on $\R^2$ is $(x_0,x_1)$ and
\[ E := |\partial_{x_0}\mathbf{x}|^2 = |\p_{x_1}\mathbf{x}|^2=:G, \quad \quad F = \p_{x_0} \mathbf{x} \cdot \p_{x_1} \mathbf{x} = 0.\]
If $\mathbf{x} : \Omega \to \R^3$ is assumed to be an immersion (recall this means that the linearization $d_p \mathbf{x}$ is injective for every $p \in \Omega$), then $E > 0$. Note that $E$ can also be thought of as a scalar factor relating the metric on $\Sigma$ induced by the pull-back of the Euclidean metric of $\R^3$ via $\mathbf{x}$ with the standard Euclidean metric of $\R^2$. Finally, we say that the conformal parametrization $\mathbf{x}$ is considered to be \textit{complete} if $\mathbf{x}^* g_{\R^3}$ induces a complete metric. Here, $g_{\R^3}$ is understood as the standard metric on $\R^3$.

Let $\xi$ be the unit normal vector of $\Sigma$ with respect to the orientation determined by $\mathbf{x} : \Omega \to \R^3$. Define
\[L = \partial^2_{x_0}\mathbf{x} \cdot \xi, \quad \quad M = \partial^2_{x_0 x_1} \mathbf{x} \cdot \xi, \quad \quad N = \partial^2_{x_1}\mathbf{x} \cdot \xi.\]
The \textit{Gaussian curvature} and the \textit{mean curvature} of $\Sigma$, respectively, are given by
\[ K  := \dfrac{LN - M^2}{E^2} = k_1 k_2, \quad \quad \quad H := \dfrac{L+N}{2E} = \dfrac{k_1+k_2}{2}.\]
Note that $k_1, k_2$ always exist and they are called the \textit{principal curvatures}. A surface is \textit{umbilic} if and only if $k_1 = k_2$. In this case, it is a standard fact that such a surface is either the plane $\R^2$ or the unit sphere $S^2$ in $\R^3$ (up to Euclidean motions). 

The quantities defined above satisfy the Mainardi-Codazzi equations. In particular, we have
\begin{equation*}
    \begin{cases}
        \partial_{x_1}L - \partial_{x_0} M  = (\partial_{x_1}E)H\\
        \partial_{x_1}M - \partial_{x_0}N = -(\partial_{x_0}E)H
    \end{cases}
\end{equation*}
Since $2EH = L + N$, the Maindardi-Codazzi equations can be rewritten as
\begin{equation*}
    \begin{cases}
        \partial_{x_0}\dfrac{L-N}{2}+\partial_{x_1}M = E\;\partial_{x_0}H\\
        \partial_{x_1}\dfrac{L-N}{2} -\partial_{x_0}M = - E \;\partial_{x_1}H
    \end{cases}
\end{equation*}
Thus, if $\Sigma$ has constant mean curvature (CMC), i.e., $H = \text{const}$, then from the Maindardi-Codazzi equations, we must have that, for $z = x_0 + i x_1$
\[ \Phi(z) = \dfrac{L-N}{2} - iM\]
is an entire function on $\R^2 = \mathbf{C}$. Typically, $\Phi$ is called the \textit{Hopf differential} of the surface $\Sigma \subset \R^3$. 

From this point onward, we consider $\Omega = \R^2$ and $\mathbf{x}: \R^2 \to \R^3$ to be a parametrization of a surface in $\R^3$ that is also a conformal immersion. This means that $E > 0$. So, $\log E$ is always well-defined. In this section, we are mainly interested in such a CMC surface with a prescribed asymptotic condition of $\log E$. In particular, we prove the following theorem.

\begin{Th}[cf. Theorem \ref{intro_asymptoticallyflatatinfintitymustbeplane}]\label{asymptoticallyflatatinfintitymustbeplane}
    Let $\mathbf{x} : \R^2 \to \R^3$ be a complete conformal smooth parametrization of a CMC surface $\Sigma \subset \R^3$ that is also an immersion. Suppose that the Gaussian curvature $K$ is uniformly bounded. Further, suppose for every $\epsilon > 0$, there are $C >0$ such that the conformal factor $E$ satisfies
    \[ 0 \leq E -1 \leq C \exp(-\epsilon|z|).\]
    Then $\Sigma$ must be $\R^2$ up to Euclidean motion. 
\end{Th}

Roughly, Theorem \ref{asymptoticallyflatatinfintitymustbeplane} says that a simply connected, noncompact surface in $\R^3$ whose metric is eventually flat at infinity must be the plane (up to Euclidean motion).

\subsection{Proof of Theorem \ref{asymptoticallyflatatinfintitymustbeplane}} We start with the following lemma about the control of the Hopf differential $\Phi$. 

\begin{Lemma}\label{controlofhopfdifferential}
    Let $\mathbf{x} : \R^2 \to \R^3$ satisfy the hypothesis of Theorem \ref{asymptoticallyflatatinfintitymustbeplane}. Then the Hopf differential $\Phi$ associated with $\mathbf{x}$ is bounded on $\R^2 =\mathbf{C}$. Consequently, $\Phi$ is constant.
\end{Lemma}

\begin{proof}
     Let $E = \exp(2u)$. By the definition of the Hopf differential in the previous subsection and the triangle inequality, we have
    \[|\Phi(z) |^2 = (H^2 - K)e^{4u} \leq (|H|^2 + |K|)e^{4u}.\]
    Since $H$ is assumed to be constant, $|K|$ is bounded, and $e^{4u}$ is bounded because $u$ approaches zero at infinity, $|\Phi(z)|$ is bounded as claimed.

    From the previous subsection, we see that if $\Sigma$ is a CMC surface, then $\Phi$ is an entire function. Since $\Phi$ is bounded. So, by Liouville's theorem, $\Phi$ must be constant by Liouville theorem.
\end{proof}

Next, we need the following technical result regarding existence of another semilinear PDE with prescribed asymptotic condition at infinity. This result might be well-known, but we include it here for completeness and self-containment.

\begin{Prop}\label{HequalzeroimpliesPhiequalzero}
    Let $\Lambda \geq 0$ be some constant. If the equation
    \[ \Delta u - \Lambda e^{-2u} = 0\]
    has a smooth solution $u$ on $\R^2$ with the property that for every $\epsilon >0$ there exists a $C>0$ such that
    \[ 0 \leq u \leq C\exp(-\epsilon|z|),\]
    then $\Lambda =0$ and in this case, the only such solution to the equation is $u\equiv 0$. 
\end{Prop}

\begin{proof}
    If $u$ is a non-negative solution to $\Delta u - \Lambda e^{-2u} = 0$ satisfying the prescribed exponential decay to zero condition at infinity, then a similar interior gradient estimate argument in Lemma \ref{exponentialdecayimpliessobolev}, we also get that $|\nabla u|$ has exponential decay rate, i.e., for every $\epsilon >0$, there exists a $C>0$ such that $|\nabla u(z)| \leq C\exp(-\epsilon|z|)$.

    Now, consider a disk $D(0,R) \subset \R^2$. Integration by parts gives us
    \[\int_{D(0,R)}\Delta u = \int_{\partial D(0,R)}\partial_r u.\]
    As a result, by the previous observation, we have
    \begin{equation}\label{eq:firstestimateforlaplacianofu}
        \int_{D(0,R)} \Delta u \leq \int_{\partial D(0,R)}|\partial_r u| \leq \int_{\partial D(0,R)} |\nabla u| \leq  2 C \pi R \,\exp(-\epsilon R).
    \end{equation}
    On the other hand, suppose $\Lambda >0$. Then since $u \to 0$ at infinity, $\Lambda e^{-2u} \to \Lambda$ at infinity, there is a large enough $R_\Lambda > 0$ such that for all $|z| > R_{\Lambda}$, we have
    \[ \Lambda - \Lambda e^{-2u} \leq |\Lambda - \Lambda e^{-2u} | <\Lambda /2.\]
    The above estimate implies
    \begin{align}\label{eq:secondestimateforlaplacianofu}
       \int_{\R^2}\Delta u \geq   \int_{\R^2\setminus D(0,R_{\Lambda})}\Delta u   &= \int_{\R^2\setminus D(0,R_{\Lambda})}\Lambda e^{-2u} \nonumber \\ 
       &> \int_{\R^2\setminus D(0,R_{\Lambda})} \Lambda/2 = \lim_{R\to\infty}\dfrac{\pi\Lambda}{2}(R^2 -R_{\Lambda}^2).
    \end{align}

    Note, \eqref{eq:firstestimateforlaplacianofu} implies that as $R \to \infty$, the integral of $\Delta u$ over $\R^2$ is identically zero. Whereas \eqref{eq:secondestimateforlaplacianofu} implies that the integral of $\Delta u$ over $\R^2$ diverges as $R$ tends to infinity. This is a contradiction. Thus, $\Lambda = 0$. Hence, $u$ is a harmonic function on $\R^2$ that tends to zero at infinity. By Liouville's theorem, $u \equiv 0$ as claimed.     
 \end{proof}

We are now ready to prove Theorem \ref{asymptoticallyflatatinfintitymustbeplane}.

\begin{proof}[Proof of Theorem \ref{asymptoticallyflatatinfintitymustbeplane}]
    From Lemma \ref{controlofhopfdifferential}, we know that $\Phi$ is constant. We argue that $\Phi\equiv 0$. Indeed, if this is the case, from the fact that 
    \[|\Phi(z) | = |k_1 - k_2| \cdot \dfrac{E}{2},\]
    and $E > 0$ (because $\mathbf{x}$ is an immersion), then $k_1 = k_2$. In other words, $\Sigma$ is umbilic. So, $\Sigma $ is either the plane or $S^2$. But $\mathbf{x} : \R^2 \to \R^3$ is complete and $\R^2$ is not conformally equivalent to $S^2$. So, $\Sigma$ must be the plane.  

    We finish the proof by showing that $\Phi \equiv 0$. Suppose otherwise and $H \neq 0$; then $k_1\neq k_2$. Without loss of generality, we assume $k_1 < k_2$. Without loss of generality, suppose $H = 1/2$. From $k_1 + k_2 = 2H = 1$, we have
    \[-1<k_1/k_2 < 1.\]
    Thus, there is a unique $\omega$ such that (this is the same argument in Section II of \cite{wente})
    \[k_1 = e^{-\omega}\sinh \omega, \quad \quad k_2 = e^{-\omega}\cosh \omega.\]
    As a result, the Gaussian curvature is given by
    \begin{equation}\label{eq:gaussian}
    K = e^{-2\omega}\sinh \omega \cosh \omega.
    \end{equation}

    Now, since $\Phi$ is constant, it follows that $(k_1 - k_2)E/2$ is also constant. Hence, there is a $c>0$ such that $E = c e^{2\omega}$. By a change of coordinates, we may suppose that $c =1$. Indeed, for a constant $t>0$, let $\Psi : \R^2 \to \R^2$ be a diffeomorphism defined by
    \[\Psi(x_0,x_1) = (tx_0, tx_1) :=(x_0',x_1').\]
    Consider the parametrization $\mathbf{x} \circ \Psi^{-1}(x_0',x_1') = \mathbf{x}(x_0(x_0',x_1'),x_1(x_0',x_1'))$. Let $E'$ be the conformal fact of the first fundamental form of $\mathbf{x} \circ \Psi^{-1}$. By the chain rule, we must have
    \[E'(x_0',x_1') = \dfrac{1}{t^2} E(x_0',x_1').\]
    So, if we choose $t = \sqrt{c}$, then we obtain
    \[E' (x_0',x_1') = \dfrac{1}{c} \cdot c \cdot \exp(2\omega(x_0',x_1')) = \exp(2\omega(x_0',x_1')).\]
    As a result, from this point onward, we assume $E = e^{2\omega}$.

    Now, Gauss' Theorem Egregium says that
    \begin{equation}\label{eq:Gaussbigtheorem}
        K = -\dfrac{\p_{x_0}(\p_{x_0}E/E)+\p_{x_1}(\p_{x_1}E/E)}{2E}.
    \end{equation}
    Using \eqref{eq:gaussian}, \eqref{eq:Gaussbigtheorem} is equivalent to 
    \[\Delta \omega + \sinh \omega \cosh \omega = 0\]
    The above equation is equivalent to
    \[\Delta (2\omega) + \sinh(2\omega) = 0.\]
    In other words, $2\omega$ is a solution to the $\sinh$-Gordon equation $\Delta u + \sinh u = 0$. By the hypothesis of $E$, we have for every $\epsilon >0$, there exists a $C>0$ such that 
    \[0\leq 2\omega \leq \log(1+C\exp(-\epsilon|z|)).\]
    By Theorem \ref{trivialsolution}, $\omega \equiv 0$. This means that $E  =1$. Therefore, the metric on $\Sigma$ is flat. But a flat, complete, noncompact, simply connected CMC surface is nothing but the plane. This leads to a contradiction. As a result, either $\Phi = 0$ or $H = 0$.

    If $\Phi =0$, then we are done. Otherwise, if $H = 0$, then set $|\Phi(z)|:=\Phi$ to be some non-negative constant. From $|\Phi(z)|^2 = (H^2 -K)e^{4u}$, where $E = \exp(2u)$, we see that
    \[K = -|\Phi|^2 e^{-4u}.\]
    Substituting this expression for the Gaussian curvature into Gauss' Theorem Egregium \eqref{eq:Gaussbigtheorem}, we obtain
    \[-|\Phi|^2 e^{-4u} = -e^{-2u}\Delta u.\]
    The above equation is, of course, equivalent to $\Delta u = |\Phi|^2 e^{-2u}$. Recall by the hypothesis of $E$, $u$ approaches to zero at exponential rate at infinity. So, we can apply Proposition \ref{HequalzeroimpliesPhiequalzero} to deduce that $\Phi = 0$ and $u \equiv 0$. Therefore, in either case, the Hopf differential must vanish, and the proof is complete.
\end{proof}

%% file: Discussion.tex
\section{Discussion}\label{Sec5}
The proof of Theorem \ref{property(E)existence} (cf. Theorem \ref{Th1.2}) relies on the algebraic procedure of turning a system of non-linear PDEs (cf. \eqref{eq:maineq}) into a (non)singular $\sinh$-Gordon \color{black} equation (cf. \eqref{eq:singulargordonsinh}). Furthermore, we require a similarity principle representation for the associated Vekua equations that appear in \eqref{eq:maineq}\color{black}. As a result, the solutions of \eqref{eq:singulargordonsinh} that go to zero at infinity \color{black} characterize solutions with exponential decay of the Seiberg-Witten vortex equations. For the Seiberg-Witten vortex equations \textit{with Higgs fields} (cf. \eqref{eq:3.12})
\begin{equation*}
\begin{cases}
\dfrac{\partial \phi}{\partial\overline{z}} = -\dfrac{1}{2} \psi_1 \overline{\psi_2},\\
i\left(\dfrac{\partial A_1}{\partial x_0} - \dfrac{\partial A_0}{\partial x_1}\right) = \dfrac{i}{2}(|\psi_1|^2 - |\psi_2|^2)\\
2\dfrac{\partial \psi_2}{\partial \overline{z}}+i(A_0 + iA_1)\psi_2 - \overline{\phi}\psi_1 = 0\\
2\dfrac{\p \psi_1}{\p z}+ i(A_0 - iA_1)\psi_1 - \phi\psi_2 = 0,
\end{cases} 
\end{equation*}
the elimination process of unknown variables is much more cumbersome and does not yield a single PDE that characterizes the solutions of \eqref{eq:3.12} in the same streamlined manner as in the case of the Seiberg-Witten vortex equations \textit{without Higgs fields}. Note the system of Vekua equations that appear above are nonhomogeneous Vekua equations. In general, these equations do not have solutions with a similarity principle representation. In the absence of a similarity principle representation, one needs to circumvent the necessity of this representation to consider the Seiberg-Witten vortex equation with Higgs field using methods from the theory of Vekua equations. \color{black}Note that in \cite{MR1952133}, it is shown that the moduli space of \eqref{eq:3.12} is non-empty.

\begin{Prop}[cf. Proposition 2.2 in \cite{MR1952133}]\label{Prop5.1}
    Let $A = iA_0 dx_0 + iA_1 dx_1$ (re-written in complex coordinate as $A = A^{1,0}dz + A^{0,1}d\overline{z}$). The moduli space of \eqref{eq:3.12} contains non-trivial solutions. Specifically, for any $c_1 \in \mathbf{C}$ and $c_2 \in \R$ such that $|c_1|=\sqrt{2} c_2$, then
    \begin{align}\label{eq:5.1}
    (A, \psi_1, \psi_2, \phi) = \left(\dfrac{-ic_2}{2}dz, c_1, c_1 e^{ic_2(z+\overline{z})}, -ic_2 e^{-ic_2(z+\overline{z})}\right)
    \end{align}
    is always a solution of \eqref{eq:3.12}.
\end{Prop}

One should compare the above with our Proposition \ref{Prop3.9}. The solutions of the type \eqref{eq:5.1} have no zeroes (unless $c_1, c_2$ are zero of course) and the connection $A$ is always flat. It is not difficult to see that to obtain zeroes, one only needs to re-scale by appropriate polynomial function in either $z$ or $\overline{z}$ variable. As a result, all of the solutions of \eqref{eq:3.12} of this type have polynomial growth. In light of our Theorem \ref{property(E)existence}, even though the method of analysis in this paper cannot be applied directly to the situation of \eqref{eq:3.12}, we make the following conjecture.

\begin{Conj}\label{Conj5.2}
    There exists smooth solution $(A, \psi_1, \psi_2, \phi)$ of \eqref{eq:3.12} such that $\psi_1, \psi_2$ satisfy the following property: For some nonnegative real constant $C$, we have
    $$0\leq C - |\psi_1|^2 \leq O(\exp(-|z|)), \quad \quad 0 \leq C - |\psi_2|^2 \leq O(\exp(-|z|)).$$
\end{Conj}

In other words, we expect that there is a subset of \color{black} solutions of \eqref{eq:3.12} that \color{black} can exhibit polynomial growth or exponential decay behavior where the connection is not necessarily flat. 

Recall that the Seiberg-Witten vortex equations (with or without Higgs fields) are derived from the Seiberg-Witten equations in dimension four. There are other generalizations of the Seiberg-Witten equations that have been considered recently (see e.g., \cite{Taubes:2016voz, MR2980921, MR1392667, haydys2015compactness, nguyen2023pin, sadegh2023threedimensionalseibergwittenequations32spinors, sadegh2024fueteroperator32spinors}). The dimensional reduction employed in this paper is not the na\"ive one where we simply forget the other dimensions. As a result, the analysis of the moduli space in the 2-dimensional theory still lends some insights about the moduli space in higher dimensions. Therefore, it is interesting to investigate the same analysis question that is considered in this paper for the dimensional reduction of the many generalizations of the Seiberg-Witten equations. We will address this direction of research in the future. 

Additionally, the systems of Vekua-type equations considered in Section \ref{preliminaries} and used in Section \ref{Sec4} to justify properties of the zero sets of solutions to vortex equations has not been used previously in \color{black} the consideration of gauge equations. This indicates that they merit further study in this setting\color{black}. 

%% file: StatementsandDeclarations.tex
\section*{Statements and Declarations}

\subsection*{Conflict of Interest}
The authors declare that they have no conflicts of interest.

\subsection*{Data Availability}
The authors declare that no data was produced. 